\documentclass[12pt,reqno]{amsart}

\usepackage{fullpage}
\usepackage{amssymb,mathrsfs}
\usepackage{bbm}

\newcommand{\diff}{\mathrm{d}}
\newcommand{\B}{\mathscr{C}}
\newcommand{\F}{\mathbb{F}}
\newcommand{\Z}{\mathbb{Z}}
\newcommand{\Q}{\mathbb{Q}}
\newcommand{\floor}[1]{\left\lfloor#1\right\rfloor}
\newcommand{\ceil}[1]{\left\lceil#1\right\rceil}
\newcommand{\dnd}{\nmid}
\newcommand{\isom}{\cong}
\newcommand{\Aut}{\mathrm{Aut}}
\newcommand{\leg}[2]{ \left( \frac{#1}{#2} \right) }

\newtheorem{thm}{Theorem}
\newtheorem{prop}[thm]{Proposition}
\newtheorem{lma}[thm]{Lemma}

\newtheorem{conj}[thm]{Conjecture}
\theoremstyle{remark}
\newtheorem*{rmk}{Remark}

\title{
Elliptic curves with a given number of points over finite fields
}

\author{Chantal David}
\address[Chantal David]{
Department of Mathematics and Statistics\\
Concordia University\\
1455 de Maisonneuve West\\
Montr\'eal, Qu\'ebec\\
H3G 1M8\\
Canada
}
\email{cdavid@mathstat.concordia.ca}
\urladdr{www.mathstat.concordia.ca/faculty/cdavid}

\author{Ethan Smith}
\address[Ethan Smith]{
Centre de recherches math\'ematiques\\
Universit\'e de Montr\'eal\\
P.O. Box 6128\\
Centre-ville Station\\
Montr\'eal, Qu\'ebec\\
H3C 3J7\\
Canada;
\and
Department of Mathematical Sciences\\
Michigan Technological University\\
1400 Townsend Drive\\
Houghton, Michigan\\
49931-1295\\
USA
}
\email{ethans@mtu.edu}
\urladdr{www.math.mtu.edu/~ethans}


\begin{document}


\begin{abstract}
Given an elliptic curve $E$ and a positive integer $N$, we consider the problem of counting
the number of primes $p$ for which the reduction of $E$ modulo $p$ possesses exactly $N$ points
over $\F_p$.  On average (over a family of elliptic curves), we show
bounds that are significantly better than what is trivially obtained by the Hasse bound.  Under
some additional hypotheses, including a conjecture concerning the short interval distribution of
primes in arithmetic progressions, we obtain an asymptotic formula for the average.
\end{abstract}

\maketitle

\section{Introduction}

Let $E$ be an elliptic curve defined over the rational field $\Q$.
For a prime $p$ where $E$ has good reduction, we let $E_p$
denote the reduced curve modulo $p$ and $\#E_p(\F_p)$
the number of $\F_p$-rational points.
Then  the trace of the Frobenius morphism at $p$, $a_p(E)$, satisfies
the well-known identity
$\#E_p(\F_p)=p+1-a_p(E)$ and the Hasse bound
$|a_p(E)|<2\sqrt p$.

Let $N$ be a positive integer. We are interested in the number of primes for
which $\#E_p(\F_p)=N$.  In particular, we are interested in the behavior of the prime
counting function
\begin{equation*}
M_E(N):=\#\{p: \#E_p(\F_p)=N\}.
\end{equation*}
Note that if $\#E_p(\F_p)=N$, then the Hasse bound implies
$(\sqrt p-1)^2<N<(\sqrt p+1)^2$,
which in turn implies that
\begin{equation*}\label{p range}
N^-:=(\sqrt N-1)^2<p<(\sqrt N+1)^2=:N^+.
\end{equation*}
Hence, $M_E(N)$ is a finite number, and we have the trivial bound
\begin{equation}\label{trivial bound}
M_E(N)\ll \frac{\sqrt N}{\log(N+1)}.
\end{equation}
In~\cite{Kow:2006}, Kowalski shows that if $E$ possesses complex multiplication (CM), then
\begin{equation}\label{Kowalski's question}
M_E(N)\ll_{E,\varepsilon}N^\varepsilon
\end{equation}
for any $\varepsilon > 0$.
He asks if the same might be true for curves without CM.
However, no bound between~\eqref{trivial bound}
and~\eqref{Kowalski's question} is known for curves without CM.

Given an integer $N$, it is always possible through a Chinese Remainder Theorem
argument to find an elliptic curve $E$ that achieves the upper bound~\eqref{trivial bound}, i.e.,
such that $\#E_p(\F_p)=N$ for every prime $p$ in the interval $(N^-,N^+)$.
Yet, for a fixed curve, one expects $M_E(N)$ to be quite small.
Consider the following na\"ive probabilistic model for $M_E(N)$.
If we suppose that the values of $\# E(\F_p)$ are uniformly distributed, i.e., that
\begin{equation} \label{naivemodel}
\mbox{Prob} \left( \# E(\F_p) = N \right) =
\begin{cases}
\frac{1}{4\sqrt{p}} &\text{if } N^-<p<N^+, \\
 0 & \text{otherwise,}
\end{cases}
\end{equation}
then we expect that
\begin{equation}\label{expect M}
\begin{split}
M_E(N)
&\approx\sum_{p} \mbox{Prob} \left( \# E(\F_p) = N \right) = \sum_{N^-<p<N^+} \frac{1}{4 \sqrt{p}} \\
&\approx\frac{1}{4 \sqrt{N}} \int_{N^-}^{N^+} \frac{\diff t}{\log{t}} \sim \frac{1}{\log{N}}.
\end{split}
\end{equation}
Moreover, it is quite easy to show (see~\cite{Kow:2006} for example) that
\begin{equation}\label{reg avg of M}
\sum_{N\le X}M_E(N)=\pi(X)+O(\sqrt X)\sim\frac{X}{\log X},
\end{equation}
where as usual, $\pi(X):=\#\{p\le X: p \text{ is prime}\}$.
Therefore, the average order of $M_E(N)$ is $1/\log N$
in accordance with the above model.
Perhaps the correct way to interpret these statements is to say that
$M_E(N)$ must be equal to zero on a density one subset of the integers for the mere reason that
the primes are a subset of the integers of density zero.
Finally, we note that while it is not difficult to see that $\liminf M_E(N)=0$, 
numerical computations~\cite{Kow:2006} are consistent with the possibility that $\limsup M_E(N)=\infty$.
In fact, using the model~\eqref{naivemodel} as in~\cite{Kow:2006}, it is possible to predict this.

\section{Statement of results}

In this paper, we study the average for $M_E(N)$  
over all elliptic curves over $\Q$ (and not over $N$ as in (\ref{reg avg of M})).
Given integers $a$ and $b$, let $E_{a,b}$ be the elliptic curve defined by the
Weierstrass equation
\begin{equation*}
E_{a,b}: y^2=x^3+ax+b.
\end{equation*}
For $A,B>0$, we define a set of Weierstrass equations by
\begin{equation*}
\B(A,B):=\{E_{a,b}: |a|\le A, |b|\le B, \Delta(E_{a,b})\ne 0\}.
\end{equation*}
The following is our first main result.

\begin{thm}\label{upper bound thm}
If $A,B\ge \sqrt N\log N$ and $AB\ge N^{3/2}(\log N)^2$, then
\begin{equation*}
\frac{1}{\#\B(A,B)}\sum_{E\in\B(A,B)}M_{E}(N)\ll \frac{\log{\log{N}}}{\log{N}}
\end{equation*}
holds uniformly for $N\ge 3$.
\end{thm}
\begin{rmk}
We refer to the expression on the left hand side of the above inequality as the average order of
$M_E(N)$ taken over the family $\B(A,B)$.  
\end{rmk}

Under an additional hypothesis  concerning the
short interval distribution of primes in arithmetic progressions, we can prove an asymptotic
formula for the average order of $M_E(N)$ over $\B(A,B)$.
In particular, we note that all of the primes counted
by $M_E(N)$ are of size $N$ lying in an interval of length $4\sqrt N$.
Therefore, we require an appropriate
short interval version of the Barban-Davenport-Halberstam Theorem.

Given real parameters $X,Y>0$ and integers $q$ and $a$,
we let $\theta(X,Y;q,a)$ denote the
weighted prime counting function
\begin{equation*}
\theta(X,Y;q,a):=\sum_{\substack{X<p\le X+Y\\ p\equiv a\pmod q}}\log p,
\end{equation*}
and we let $E(X,Y;q,a)$ be the error in approximating $\theta(X,Y;q,a)$ by $Y/\varphi(q)$.
That is,
\begin{equation*}
E(X,Y;q,a):=\theta(X,Y;q,a)-\frac{Y}{\varphi(q)}.
\end{equation*}

\begin{conj}\label{bdh conj}
(Barban-Davenport-Halberstam for intervals of length $X^\eta$)
\label{shortBDH} Let $0 < \eta \leq 1$, and let $\beta > 0$ be arbitrary. Suppose
that $X^\eta \leq Y \leq X$, and that $Y / (\log{X})^\beta \leq Q \leq Y.$
Then
\begin{align*}
\sum_{q \leq Q}\sum_{\substack{a=1\\ (a,q)=1}}^{q}\left| E(X,Y;q,a)\right|^2 \ll YQ \log{X}.
\end{align*}
\end{conj}
\begin{rmk}
If $\eta=1$, this is essentially the classical Barban-Davenport-Halberstam Theorem.
See for example~\cite[p.~196]{Dav:1980}.
The best results known are due to Languasco, Perelli, and Zaccagnini~\cite{LPZ:2010} who
show that for any $\epsilon>0$, Conjecture~\ref{shortBDH} holds unconditionally for
$\eta = 7/12 + \epsilon$ and for $\eta=1/2+\epsilon$ under the
Generalized Riemann Hypothesis.
For our application, we essentially need $\eta=1/2-\epsilon$.
\end{rmk}

\begin{thm}\label{main thm}
Let $\gamma>0$,
and assume that Conjecture~\ref{bdh conj} holds with
$\eta=\frac{1}{2}-(\gamma+2)\frac{\log\log N}{\log N}$.
 Suppose further that $A,B\ge \sqrt N(\log N)^{1+\gamma}\log\log N$ and that
$AB\ge N^{3/2}(\log N)^{2+\gamma}\log\log N$.
Then for any odd integer $N$, we have
\begin{equation*}
\frac{1}{\#\B(A,B)}\sum_{E\in\B(A,B)}M_{E}(N)
=K(N)\frac{N}{\varphi(N)\log N}+O\left(\frac{1}{(\log N)^{1+\gamma}}\right),
\end{equation*}
where
\begin{equation*}
K(N):=
	\prod_{\ell \dnd N}\left(
	1-\frac{\leg{N-1}{\ell}^2\ell+1}{(\ell-1)^2(\ell+1)}
	\right)
	\prod_{\substack{\ell\mid N\\ 2\dnd\nu_\ell(N)}}
	\left(1-\frac{1}{\ell^{\nu_\ell(N)}(\ell-1)}\right)
	\prod_{\substack{\ell\mid N\\ 2\mid\nu_\ell(N)}}
	\left(1-\frac{\ell-\leg{-N_\ell}{\ell}}
	{\ell^{\nu_\ell(N)+1}(\ell-1)}\right),
\end{equation*}
$\nu_\ell$ denotes the usual $\ell$-adic valuation, and
$N_\ell:=N/\ell^{\nu_\ell(N)}$ denotes the $\ell$-free part of $N$.
\end{thm}
\begin{rmk}
We note that $K(N)$ is uniformly bounded as a function of $N$.
We also note that $N/\varphi(N)\ll\log\log N$ (see~\cite[Theorem 328]{HW:1979} for example), which
gives the upper bound of Theorem \ref{upper bound thm}.
Working with V. Chandee and D. Koukoulopoulos, the authors have recently shown that the upper bound 
implicit in Theorem~\ref{main thm} holds unconditionally.  That is, Theorem~\ref{upper bound thm} holds 
with $\log\log N$ replaced by $N/\varphi(N)$.
\end{rmk}

The average of Theorem~\ref{main thm} displays some
interesting characteristics that are not present in the average order~\eqref{reg avg of M}.
In particular, the main term of the average in Theorem~\ref{main thm}
does not depend solely on the size of the integer $N$ but also on some arithmetic properties of
$N$ as it involves the factor $K(N)N/\varphi(N)$.
The occurrence of the weight $\varphi(N)$ appearing in the denominator  seems to suggest that
this is another example of the Cohen-Lenstra Heuristics~\cite{CL:1984,CL:1984-2},
which predict that random groups $G$ occur with probability weighted by $1/\#\Aut(G)$.
Notice that if as an additive group $E(\F_p) \simeq\Z/N\Z$,
then $\#\Aut(E(\F_p)) = \varphi(N)$.  Indeed, the Cohen-Lenstra Heuristics predict that
relative to other groups of same size, the cyclic groups are the most likely to occur since they have the 
fewest number of automorphisms.

In some recent work, the authors explored this connection further by
considering the average of
\begin{equation*}
M_E(G) := \# \left\{ p : E(\F_p) \simeq G \right\}
\end{equation*}
for those Abelian groups $G$ which may arise as the group of $\F_p$-rational points of
an elliptic curve.  
This is the subject of a forthcoming paper~\cite{DS2}.
Given an elliptic curve $E$, it is well-known that
\begin{equation*}
E(\F_p) \simeq \Z/N_1 \Z \times \Z / N_1 N_2 \Z,
\end{equation*}
for some positive integers $N_1, N_2$ satisfying the Hasse bound: $|p+1-N_1^2 N_2| \leq 2 \sqrt{p}$.
Under Conjecture~\ref{bdh conj}, it is shown in~\cite{DS2} that
for every odd order group $G = \Z/N_1 \Z \times \Z/N_1 N_2 \Z$, we have that
\begin{equation*}
\frac{1}{\#\B(A,B)}\sum_{E\in\B(A,B)}M_E(G) \sim
	 K(G) \frac{\# G}{\# \Aut(G)\log(\#G)},
\end{equation*}
provided that $A,B$, and the exponent of $G$ (the size of the largest cyclic subgroup) are large enough with 
respect to $\#G=N_1^2 N_2$.  The function $K(G)$ is explicitly computed and shown to be 
non-zero and absolutely bounded as a function of $G$.


We can express the results of Theorem~\ref{main thm} as stating that for a ``random curve"
$E/\Q$ and a ``random prime" $p \in (N^-, N^+)$,
\begin{eqnarray*}
\mathrm{Prob}\left(\#E(\F_p) = N \right)
\approx \frac{K(N) \frac{N}{\varphi(N) \log(N)}}{ \frac{4 \sqrt{N}}{\log{N}}}
= \frac{K(N) N}{\varphi(N)} \frac{1}{4\sqrt{N}}
\end{eqnarray*}
refining the na\"ive model given by~\eqref{naivemodel}.
Here as in~\eqref{naivemodel}, we make the assumption that there are about
$4\sqrt N/\log N$ primes in the interval $(N^-,N^+)$ though we can not justify such an
assumption even under the Riemann Hypothesis.

There are many open conjectures about the distributions of invariants associated with the 
reductions of a fixed elliptic curve over the finite fields $\F_p$ such as the famous conjectures 
of Koblitz~\cite{Kob:1988} and of Lang and Trotter~\cite{LT:1976}.
The Koblitz Conjecture concerns the number of primes $p\le X$ such that $\#E(\F_p)$ is prime.
The fixed trace Lang-Trotter Conjecture concerns the number of primes $p\le X$ such that
the trace of Frobenius $a_p(E)$ is equal to a fixed integer $t$.  Another 
conjecture of Lang and Trotter (also called the Lang-Trotter Conjecture) 
concerns the number of primes
$p\le X$ such that the Frobenius field $\Q(\sqrt{a_p(E)^2-4p})$ is a fixed 
imaginary quadratic field $K$.

These conjectures are all completely open.
To gain evidence, it is natural to consider the averages for these conjectures over 
some family of elliptic curves.
This has been done by various authors originating with the work of Fouvry and 
Murty~\cite{FM:1996} for the number of supersingular primes (i.e., the fixed trace Lang-Trotter 
Conjecture for $t=0$).  See~\cite{DP:1999,DP:2004,Jam:2004,BBIJ:2005,JS:2011,CFJKP:2011} 
for other averages regarding the fixed trace Lang-Trotter Conjecture.
The average order for the Koblitz Conjecture was considered in~\cite{BCD:2011}.
Very recently, the average has been successfully carried out for the Lang-Trotter Conjecture
on Frobenius fields~\cite{CIJ:pp}.
The average order that we consider in this paper displays a very different character than 
the above averages.   This is primarily because the size of primes considered varies with the 
parameter $N$. Moreover, they all must lie in a very short interval.
This necessitates the use of a short interval version of the Barban-Davenport-Halberstam Theorem 
(Conjecture~\ref{bdh conj} above).
This is also the first time that one observes a Cohen-Lenstra phenomenon governing the 
distribution of the average.


\section{Acknowledgement}
The authors would like to thank K. Soundararajan for pointing out how to improve
the average upper bound of Theorem \ref{upper bound thm} by using short Euler products
holding for almost all characters. They also thank Henri Cohen, Andrew Granville and
Dimitris Koukoulopoulos for useful discussions related to this work and Andrea  Smith for a
careful reading of the manuscript.

\section{Reduction to an average of class numbers}

Given a (not necessarily fundamental) discriminant $D<0$, we follow Lenstra~\cite{Len:1987}
in defining the \textit{Kronecker class number} of discriminant $D$ by
\begin{equation}\label{defn of K-class no}
H(D):=\sum_{\substack{f^2|D\\ \frac{D}{f^2}\equiv 0,1\pmod{4}}}\frac{h(D/f^2)}{w(D/f^2)},
\end{equation}
where $h(d)$ denotes the (ordinary) class number of the unique imaginary quadratic
order of discriminant $d<0$ and $w(d)$ denotes the cardinality of its unit group.
\begin{thm}[Deuring]\label{deuring}
Let $p>3$ be a prime and $t$ an integer such that $t^2-4p<0$.
Then
\begin{equation*}
\sum_{\substack{\tilde{E}/\F_p\\ a_p(E)=t}}\frac{1}{\#\mathrm{Aut}(E)}
=H(t^2-4p),
\end{equation*}
where the sum is over the $\F_p$-isomorphism classes of elliptic curves.
\end{thm}
\begin{proof}
See~\cite[p.~654]{Len:1987}.
\end{proof}

The first step in computing the average order of $M_E(N)$ over $\B(A,B)$ is to reduce to an
average of class numbers by using Deuring's Theorem. The following estimate
will then be crucial to obtain the upper bound of Theorem~\ref{upper bound thm}, and is also used
in getting an optimal average length in Theorem~\ref{main thm}.

\begin{prop}\label{upper bound for H on avg}
For primes $p$ in the range $N^-<p<N^+$, we define the quadratic polynomial
\begin{equation}\label{defn of D}
D_N(p):=(p+1-N)^2-4p=p^2-2(N+1)p+(N-1)^2.
\end{equation}
Then
\begin{equation} \label{sumofthethm}
\sum_{N^-<p<N^+}H\left(D_N(p)\right)\ll \frac{N \log{\log{N}}}{\log{N}}.
\end{equation}
\end{prop}

Before giving the proof of this result, we define
some notation that we will use throughout the remainder of the article.
Given a negative discriminant $d$, we write $\chi_d$ for the Kronecker symbol $\leg{d}{\cdot}$.
Since $d$ is not a perfect square, the Dirichlet $L$-series defined by
\begin{equation*}
L(s,\chi_d):=\sum_{n=1}^\infty\frac{\chi_d(n)}{n^s}
\end{equation*}
converges at $s=1$.
Finally, given a positive integer $f$, we let
\begin{equation}\label{defn of d}
d_{N,f}(p):=\frac{D_N(p)}{f^2},
\end{equation}
where $D_N(p)$ is defined by~\eqref{defn of D}.

\begin{proof}[Proof of Proposition~\ref{upper bound for H on avg}.]
The definition of the Kronecker class number~\eqref{defn of K-class no} and the
class number formula~\cite[p.~515]{IK:2004}
\begin{equation}\label{class no formula}
\frac{h(d)}{w(d)}=\frac{\sqrt{|d|}}{2\pi}L(1,\chi_d)
\end{equation}
give us the identity
\begin{equation*}
\sum_{N^-<p<N^+} H(D_N(p))
= \sum_{N^-<p<N^+}
\sum_{\substack{f^2\mid D_N(p)\\ \frac{D_N(p)}{f^2}\equiv 0,1\pmod 4}}
	\frac{\sqrt{|D_N(p)|}}{2\pi f}L(1,\chi_{d_{N,f}(p)}).
\end{equation*}
Since $|D_N(p)|\le 4N$, this yields
\begin{equation}\label{use bound for D}
\begin{split}
\sum_{N^-<p<N^+} H(D_N(p))
&\ll\sqrt N\sum_{N^-<p<N^+}
\sum_{\substack{f^2\mid D_N(p)\\ \frac{D_N(p)}{f^2}\equiv 0,1\pmod 4}}
	\frac{L(1,\chi_{d_{N,f}(p)})}{f}.
\end{split}
\end{equation}

In order to obtain an optimal bound for this expression, we will use the fact that
for almost all primitive Dirichlet characters $\psi$, $L(1,\psi)$ is well-approximated by a very
short Euler product.  More precisely, fix any integer $\alpha\ge1$.
Then by~\cite[Proposition 2.2]{GS:2003}, we know that
\begin{equation} \label{shortEP}
L(1, \psi) = \prod_{\ell \leq (\log{Q})^\alpha} \left( 1 - \frac{\psi(\ell)}{\ell} \right)^{-1}
	 \left( 1 + \underline{o}(1) \right)
\end{equation}
for all but at most $Q^{2/\alpha + 5 \log{\log{\log{Q}}}/\log{\log{Q}}}$ of the primitive characters of conductor less than $Q$. We remark that this gives a good upper bound for
$L(1,\chi)$ whenever $\chi$ is a Dirichlet character modulo $q\le Q$ which is induced by
a primitive character $\psi$ satisfying~\eqref{shortEP}.
Indeed, let $\chi$ be such a Dirichlet character, and let $\psi$ be the primitive character that
induces $\chi$. Then by~\eqref{shortEP}, we have
\begin{equation*}
\begin{split}
L(1, \chi)
&= \prod_{\ell \mid q} \left( 1 - \frac{\psi(\ell)}{\ell} \right)
 \prod_{\ell\le (\log{Q})^\alpha}\left( 1-\frac{\psi(\ell)}{\ell}\right)^{-1}\left(1+\underline{o}(1) \right)\\
&=\prod_{\substack{\ell\mid q\\ \ell > (\log{Q})^\alpha}} \left(1-\frac{\psi(\ell)}{\ell} \right)
\prod_{\ell \leq (\log{Q})^\alpha} \left( 1 - \frac{\chi(\ell)}{\ell} \right)^{-1}
\left(1+\underline{o}(1) \right)\\
&\ll\prod_{\substack{\ell\mid q\\ \ell > (\log{Q})^\alpha}} \left(1+\frac{1}{\ell} \right)
\prod_{\ell \leq (\log{Q})^\alpha} \left( 1 - \frac{1}{\ell} \right)^{-1}\\
&\ll\prod_{\substack{\ell\mid q\\ \ell > (\log{Q})^\alpha}} \left(1+\frac{1}{\ell} \right)\log\log Q,
\end{split}
\end{equation*}
where the last line follows by Mertens' formula~\cite[p.~34]{IK:2004} since $\alpha$ is fixed.
For the remaining product, we observe that
\begin{equation*}
\prod_{\substack{\ell\mid q\\ \ell > (\log{Q})^\alpha}} \left(1+\frac{1}{\ell} \right)
\le\exp\left\{\sum_{\substack{\ell\mid q\\ \ell>(\log Q)^\alpha}}\frac{1}{\ell}\right\}
\le\exp\left\{\frac{\omega(q)}{(\log Q)^{\alpha}}\right\}
\le\exp\left\{\frac{(\log q)^{1-\alpha}}{\log 2}\right\},
\end{equation*}
where $\omega(q)$ denotes the number of distinct prime factors of $q$.
Therefore, since $\alpha\ge1$,  we may conclude that
if $\chi$ is a character of modulus $q\le Q$ and~\eqref{shortEP} holds for the primitive
character inducing $\chi$, then
\begin{equation}\label{nexcept-L bound}
L(1,\chi)\ll\log\log Q.
\end{equation}

We make use of this fact in~\eqref{use bound for D} as follows.
Let $d_N^*(p)$ be the discriminant of the imaginary
quadratic field $\Q(\sqrt{D_N(p)})$.  Then $d_N^*(p)$ is a fundamental discriminant,
and $\chi_{d_N^*(p)}$ is the primitive character inducing every character of the set
$\{\chi_{d_{N,f}(p)}: f^2\mid D_N(p)\}$.  Furthermore, $|d_N^*(p)|$ is the conductor of
each of these characters, and $3\le |d_N^*(p)|\le 4N$.
Now fix some $\alpha>100$, and
let $\mathscr E(Q)$ be the set of primitive characters of conductor less than or equal to $Q$
for which~\eqref{shortEP} does not hold.  Then $\# \mathscr E(4N)\ll N^{1/50}$.
We now divide the outer sum over $p$ on the right-hand side of ~\eqref{use bound for D} 
according to whether or not the primitive character $\chi_{d_N^*(p)}$ is in the
exceptional set $\mathscr E(4N)$.  For those $p$ for which $\chi_{d_N^*(p)}$ is not
exceptional, we use~\eqref{nexcept-L bound}, writing
\begin{equation*}
\sum_{\substack{N^-<p<N^+\\ \chi_{d_N^*(p)} \not\in \mathscr{E}(4N)}}
\sum_{\substack{f^2\mid D_N(p)\\ \frac{D_N(p)}{f^2}\equiv 0,1\pmod 4}}
	\frac{L(1,\chi_{d_{N,f}(p)})}{f}
\ll\log{\log{N}}\sum_{f\le 2\sqrt N}\frac{1}{f}\sum_{\substack{N^-<p<N^+\\  {f^2\mid D_N(p)}}} 1.
\end{equation*}
To bound the sum over $p$, we apply the Brun-Titchmarsh inequality~\cite[p.~167]{IK:2004},
which gives that
\begin{equation*}
\#\{N^-<p<N^+ : p\equiv a\pmod f\}
\ll\frac{\sqrt N}{\varphi(f)\log(4\sqrt N/f)}.
\end{equation*}
For the sum over $a$, we use the bound
\begin{equation*}
\#\{a\in\Z/f\Z: D_N(a)\equiv 0\pmod f\}\ll\sqrt f,
\end{equation*}
which is Lemma~\ref{bound quad cong soln count} of
Section~\ref{proofs of lemmas}.
Combining these two estimates, we have that
\begin{equation}\label{nexcept-sum bound}
\begin{split}
\sum_{\substack{N^-<p<N^+\\ \chi_{d_N^*(p)} \not\in \mathscr{E}(4N)}}
\sum_{\substack{f^2\mid D_N(p)\\ \frac{D_N(p)}{f^2}\equiv 0,1\pmod 4}}
	\frac{L(1,\chi_{d_{N,f}(p)})}{f}
&\ll \frac{\sqrt N\log{\log{N}}}{\log{N}} \sum_{f\ge 1}\frac{\log f}{f^{1/2}\varphi(f)}\\
&\ll\frac{\sqrt N\log{\log{N}}}{\log{N}}.
\end{split}
\end{equation}

It remains to estimate the sum over primes $p$ such that
$\chi_{d_N^*(p)} \in \mathscr{E}(4N)$.
In that case, we simply need the standard bound
\begin{equation*}
L(1, \chi) \ll \log{q},
\end{equation*}
which is valid for all Dirichlet characters of conductor $q$ 
(see~\cite[p.~96]{Dav:1980} for example).
We note that
\begin{equation*}
\#\{N^-<p<N^+: \chi_{d_N^*(p)}\in\mathscr E(4N)\}\le\# \mathscr E(4N)\tau(4N),
\end{equation*}
where $\tau(n)$ denotes the number of positive divisors of $n$.
Therefore, we obtain the bound
\begin{equation}\label{except-sum bound}
\begin{split}
\sum_{\substack{N^-<p<N^+\\ \chi_{d_N^*(p)}\in \mathscr{E}(4N)}}
\sum_{\substack{f^2\mid D_N(p)\\ \frac{D_N(p)}{f^2}\equiv 0,1\pmod 4}}
	\frac{L(1,\chi_{d_{N,f}(p)})}{f}
&\ll\log{N} \sum_{\substack{N^-<p<N^+\\ \chi_{d_N^*(p)}\in\mathscr{E}(4N)}}
	\sum_{f^2 \mid D_N(p)} \frac{1}{f}\\
&\ll N^{1/50+\varepsilon} 
\end{split}
\end{equation}
for any $\varepsilon > 0$.
Combining~\eqref{use bound for D},~\eqref{nexcept-sum bound}, and~\eqref{except-sum bound}
completes the proof of Proposition~\ref{upper bound for H on avg}.
\end{proof}

\begin{prop}\label{average order in terms of class numbers}
Let $D_N(p)$ be as defined by~\eqref{defn of D}. Then
\begin{equation*}
\frac{1}{\#\B(A,B)}\sum_{E\in\B(A,B)}M_{E}(N)
=\sum_{N^-<p<N^+}\frac{H(D_N(p))}{p}+\mathcal E(N;A,B),
\end{equation*}
where
\begin{equation*}
\mathcal E(N;A,B)\ll
\frac{\log\log N}{N\log N}+\left(\frac{1}{A}+\frac{1}{B}\right)\sqrt N\log\log N+\frac{N^{3/2}\log N\log\log N}{AB}
\end{equation*}
uniformly for $A,B>0$.
\end{prop}
\begin{rmk}
The above holds without assuming that $N$ is odd.
\end{rmk}

\begin{proof}[Proof of Proposition~\ref{average order in terms of class numbers}]
First, we write $M_E(N)$ as a sum over primes and interchange sums
to obtain
\begin{equation*}
\begin{split}
\frac{1}{\#\B(A,B)}
\sum_{E\in\B(A,B)}M_{E}(N)
&=
\frac{1}{\#\B(A,B)}
\sum_{N^-<p<N^+}
\sum_{\substack{E\in\B(A,B)\\ \#E_p(\F_p)=N}}1.\\
\end{split}
\end{equation*}
For each prime $p$, we group the $E\in\B(A,B)$ according to which isomorphism
class they reduce modulo $p$, writing
\begin{equation*}
\sum_{\substack{E\in\B(A,B)\\ \#E_p(\F_p)=N}}1
=\sum_{\substack{\tilde E/\F_p\\ \#\tilde E(\F_p)=N}}
	\#\{E\in\B(A,B): E_p\isom\tilde E\}.
\end{equation*}
For $N$ large enough ($N\ge 8$), the primes $2$ and $3$ will not enter into
the sum over $p$.  Thus, we will assume
that $p>3$ throughout the remainder of the article.
Therefore, given a elliptic curve defined over $\F_p$, we may associate
a Weierstrass equation, say $E_{s,t}: y^2=x^3+sx+t$ with $s,t\in\F_p$.
Using a character sum argument as in~\cite[pp.~93-95]{FM:1996}, we have
\begin{equation*}
\begin{split}
\#\{E\in\B(A,B): E_p\isom E_{s,t}\}
=\frac{4AB}{\#\Aut(E_{s,t})p}
&+O\left(\frac{AB}{p^2}+\sqrt p(\log p)^2\right)\\
&+\begin{cases}
O\left(\frac{A\log p}{\sqrt p}+\frac{B\log p}{\sqrt p}\right)&\text{if }st\ne 0,\\
O\left(\frac{A\log p}{\sqrt p}+B\log p\right)&\text{if }s=0,\\
O\left(A\log p+\frac{B\log p}{\sqrt p}\right)&\text{if }t=0.
\end{cases}
\end{split}
\end{equation*}
Here $\Aut(E_{s,t})$ denotes the size of the automorphism group of $E_{s,t}$ over $\F_p$.
Substituting this estimate and applying Theorem~\ref{deuring}, we find that
\begin{equation*}
\frac{1}{\#\B(A,B)}
\sum_{E\in\B(A,B)}M_{E}(N)
=
\sum_{N^-<p<N^+}\frac{H(D_N(p))}{p}+\mathcal E(N;A,B),
\end{equation*}
where
\begin{equation*}
\begin{split}
\mathcal E(N;A,B)
&\ll\left\{\frac{1}{N^2}
	+\frac{\log N}{\sqrt N}\left(\frac{1}{A}+\frac{1}{B}\right)
	+\frac{\sqrt N(\log N)^2}{AB}\right\}
	\sum_{N^-<p<N^+}H(D_N(p))\\
&+\left(\frac{1}{A}+\frac{1}{B}\right)\log N\sum_{N^-<p<N^+}1.
\end{split}
\end{equation*}
In estimating the error, we have used the facts that $\#\B(A,B)=4AB+O(A+B)$, $p=N+O(\sqrt N)$ for every
prime $p$ in the interval $(N^-,N^+)$, and there are at most $10$ isomorphism classes $E_{s,t}$ over 
$\F_p$ with $st=0$.
The result now follows by applying Proposition~\ref{upper bound for H on avg} and the 
upper bound $\{N^-<p<N^+\}\ll\sqrt N/\log N$.
\end{proof}

Theorem~\ref{upper bound thm} now follows immediately upon combining
Proposition ~\ref{upper bound for H on avg} and
Proposition~\ref{average order in terms of class numbers}, and noting again that
$p=N+O(\sqrt N)$ for every prime $p$ in the interval
$(N^-,N^+)$.
Furthermore, we have reduced the proof of Theorem~\ref{main thm} to computing
the sum
\begin{equation*}
\sum_{N^-<p<N^+}\frac{H(D_N(p))}{p}.
\end{equation*}
This computation requires several intermediate results.
Therefore, we delay it until Section~\ref{proof of main thm}.

\section{A short average of special values of Dirichlet $L$-functions}

Since Theorem~\ref{main thm} holds only for odd $N$, we assume for the
the remainder of the article that $N$ is an odd integer except during
the proof of Lemma~\ref{bound quad cong soln count}.  Recall that
Lemma~\ref{bound quad cong soln count} was used in the proof of
Proposition~\ref{upper bound for H on avg}, which holds for all positive integers $N$.

In this section, we prove a short average result for special values of Dirichlet
$L$-functions that is needed to compute our average over elliptic curves. As the average
is very short, we need that the equivalent of the Barban-Davenport-Halberstam Theorem holds
for intervals of that size.  

\begin{thm} \label{thmaverageSV}
Let $\gamma>0$.
Suppose that $N^-\le X<X+Y\le N^+$ with $Y\gg \sqrt N/(\log N)^\upsilon$ for some choice of 
$\upsilon\ge 0$.
Assume that Conjecture~\ref{shortBDH} holds for intervals of length $Y$.
Then for odd integers $N$,
\begin{equation*}
\sum_{\substack{f\le 2\sqrt{X+Y}\\ (f,2)=1}}\frac{1}{f}
\sum_{\substack{X<p\le X+Y\\ f^2\mid D_N(p)}}
L\left(1,\chi_{d_{N,f}(p)}\right)\log p = K_0(N)Y + O \left( \frac{Y}{(\log{N})^\gamma} \right),
\end{equation*}
where
\begin{equation}\label{K(N) as a double infinite sum}
K_0(N):=
\sum_{\substack{f=1\\ (f,2)=1}}^\infty\frac{1}{f}\sum_{n=1}^\infty\frac{1}{n\varphi(4nf^2)}
	\sum_{\substack{a\in\Z/4n\Z\\ a\equiv 1\pmod 4}}\leg{a}{n}
	\#C_N(a,n,f),
\end{equation}
and
\begin{equation}\label{defn of C}
C_N(a,n,f):=\left\{z\in(\Z/4nf^2\Z)^*: D_N(z)\equiv af^2\pmod{4nf^2}\right\}.
\end{equation}
\end{thm}

\begin{proof}
Let $U$ be a real parameter to be determined.
Using partial summation and Burgess' bound for character sums~\cite[Theorem 2]{Bur:1963} to
bound the tail of the $L$-series, we have
\begin{equation*}
L\left(1,\chi_{d_{N,f}(p)}\right)
=\sum_{n\ge 1}\leg{d_{N,f}(p)}{n}\frac{1}{n}
=\sum_{n\le U}\leg{d_{N,f}(p)}{n}\frac{1}{n}
	+O\left(\frac{|d_{N,f}(p)|^{\frac{7}{32}}}{\sqrt U}\right).
\end{equation*}
For $N^-<p<N^+$, we have $|d_{N,f}(p)|\le 4N/f^2$, and hence
\begin{equation*}
\begin{split}
\sum_{\substack{f\le 2\sqrt{X+Y}\\ (f,2)=1}}\frac{1}{f}
\sum_{\substack{X<p\le X+Y\\ f^2\mid D_N(p)}}\log p
\sum_{n> U}\frac{1}{n}
\leg{d_{N,f}(p)}{n}
\ll\frac{YN^{7/32}}{\sqrt U}.
\end{split}
\end{equation*}
Now let $V$ be a real parameter to be determined.
Using Lemma~\ref{bound quad cong soln count}, we obtain
\begin{equation*}
\begin{split}
\sum_{\substack{V<f\le 2\sqrt{X+Y}\\ (f,2)=1}}\frac{1}{f}
\sum_{n\le U}&
\frac{1}{n}
\sum_{\substack{X<p\le X+Y\\ f^2\mid D_N(p)}}
\leg{d_{N,f}(p)}{n}\log p\\
&\ll\log U\log N\sum_{\substack{V<f\le 2\sqrt{X+Y}\\ (f,2)=1}}
	\frac{1}{f}\sum_{\substack{X<k\le X+4\sqrt N\\ f|D_N(p)}}1\\
&\ll\log U\log N\sum_{\substack{V<f\le 2\sqrt{X+Y}\\ (f,2)=1}}
\frac{\sqrt N\#\{z\in\Z/f\Z: D_N(z)\equiv 0\pmod f\}}{f^2}\\
&\ll\sqrt N\log U\log N\sum_{f>V}\frac{1}{f^{3/2}}\\
&\ll\frac{\sqrt N\log U\log N}{\sqrt V},
\end{split}
\end{equation*}
and therefore,
\begin{equation*}
\begin{split}
\sum_{\substack{f\le 2\sqrt{X+Y}\\ (f,2)=1}}\frac{1}{f}
\sum_{\substack{X<p\le X+Y\\ f^2\mid D_N(p)}}
L\left(1,\chi_{d_{N,f}(p)}\right)\log p
&=\sum_{\substack{f\le V\\ (f,2)=1}}\frac{1}{f}
\sum_{n\le U}
\frac{1}{n}
\sum_{\substack{X<p\le X+Y\\ f^2\mid D_N(p)}}
\leg{d_{N,f}(p)}{n}\log p\\
&\quad+O\left(\frac{YN^{7/32}}{\sqrt U}+\frac{\sqrt N\log U\log N}{\sqrt V}\right).
\end{split}
\end{equation*}
With $C_N(a,n,f)$ as defined by~\eqref{defn of C},
we regroup terms on the right hand side above, writing
\begin{equation*}
\begin{split}
\sum_{\substack{f\le V, n\le U\\ (f,2)=1}}\frac{1}{fn}
\sum_{\substack{X<p\le X+Y\\ f^2\mid D_N(p)}}
\leg{d_{N,f}(p)}{n}\log p
&=\sum_{\substack{f\le V, n\le U\\ (f,2)=1}}\frac{1}{fn}
\sum_{\substack{a\in\Z/4n\Z\\ a\equiv 1\pmod 4}}\leg{a}{n}
\sum_{b\in C_N(a,n,f)}\theta(X,Y;4nf^2,b)\\
&\quad+O\left(\sum_{\substack{f\le V, n\le U\\ (f,2)=1}}\frac{1}{fn}
\sum_{\substack{a\in\Z/4n\Z\\ a\equiv 1\pmod 4}}\leg{a}{n}
\sum_{\substack{\substack{X<p\le X+Y\\ f^2\mid D_N(p)}\\ D_N(p)\equiv af^2\\ p\mid 4nf^2}}\log p\right).
\end{split}
\end{equation*}
If $p$ satisfies the congruence
$D_N(p)\equiv af^2\pmod{4nf^2}$ and $p$ divides $4nf^2$, then $p$ divides
$(4nf^2,(N-1)^2-af^2)$.  It follows that such a $p$ must divide $n(N-1)$.
Hence, the error term
\begin{equation*}
\sum_{\substack{f\le V, n\le U\\ (f,2)=1}}\frac{1}{fn}
\sum_{\substack{a\in\Z/4n\Z\\ a\equiv 1\pmod 4}}\leg{a}{n}
\sum_{\substack{\substack{X<p\le X+Y\\ f^2\mid D_N(p)}\\ D_N(p)\equiv af^2\\ p\mid 4nf^2}}\log p
\ll U\log N\log V + U\log U\log V,
\end{equation*}
and
\begin{equation*}
\begin{split}
\sum_{\substack{f\le 2\sqrt{X+Y}\\ (f,2)=1}}\frac{1}{f}
\sum_{\substack{X<p\le X+Y\\ f^2\mid D_N(p)}}
L\left(1,\chi_{d_{N,f}(p)}\right)\log p
&=\sum_{\substack{f\le V, n\le U\\ (f,2)=1}}\frac{1}{fn}
\sum_{\substack{a\in\Z/4n\Z\\ a\equiv 1\pmod 4}}\leg{a}{n}
\sum_{b\in C_N(a,n,f)}\theta(X,Y;4nf^2,b)\\
&\quad+O\left(
		\frac{YN^{7/32}}{\sqrt U}+\frac{\sqrt N\log U\log N}{\sqrt V}
		+U\log(UN)\log V\right).
\end{split}
\end{equation*}
We approximate $\theta(X,Y;4nf^2,b)$ by $Y/\varphi(4nf^2)$, incurring an error of
\begin{equation}\label{error from approximating sum over primes}
\sum_{\substack{f\le V,n\le U\\ (f,2)=1}}\frac{1}{fn}
\sum_{\substack{a\in\Z/4n\Z\\ a\equiv 1\pmod 4}}\leg{a}{n}
\sum_{b\in C_N(a,n,f)}E(X,Y; 4nf^2,b).
\end{equation}
For any given value of $b\in(\Z/4nf^2\Z)^*$, there is at most one value of
$a\in\Z/4n\Z$ satisfying the congruence $D_N(b)\equiv af^2\pmod{4nf^2}$.  Hence, interchanging the two inner sums shows that
\begin{equation*}
\sum_{\substack{a\in\Z/4n\Z\\ a\equiv 1\pmod 4}}\leg{a}{n}
\sum_{b\in C_N(a,n,f)}E(X,Y; 4nf^2,b)
\ll \sum_{b\in(\Z/4nf^2\Z)^*}\left|E(X,Y;4nf^2,b)\right|.
\end{equation*}
By Cauchy-Schwarz,
the error term~\eqref{error from approximating sum over primes} is
\begin{equation*}
\begin{split}
&\ll
\sum_{f\le V}\frac{1}{f}
\sum_{n\le U}
\frac{1}{n}
\sum_{b\in(\Z/4nf^2\Z)^*}\left|E(X,Y;4nf^2,b)\right|\\
&\le
\sum_{f\le V}\frac{1}{f}
\left[
\sum_{n\le U}\frac{\varphi(4nf^2)}{n^2}
\right]^{1/2}
\left[
\sum_{n\le U}
\sum_{b\in(\Z/4nf^2\Z)^*}
\left|E(X,Y;4nf^2,b)\right|^2
\right]^{1/2}\\
&\ll
V\sqrt{\log U}
\left[
\sum_{q\le 4UV^2}\sum_{\substack{a=1\\ (a,q)=1}}^q
\left|E(X,Y;q,a)\right|^2\right]^{1/2}.
\end{split}
\end{equation*}
Assuming Conjecture~\ref{shortBDH} for an appropriate value of $\eta$, we obtain the bound
\begin{equation*}
V\sqrt{\log U}
\left[
\sum_{q\le 4UV^2}\sum_{\substack{a=1\\ (a,q)=1}}^q
\left|E(X,Y;q,a)\right|^2\right]^{1/2}
\ll\frac{YV\sqrt{\log U\log N}}{(\log N)^{2\upsilon+3\gamma+5}}
\end{equation*}
whenever
\begin{equation}\label{bdh restriction on U and V}
UV^2\le \frac{Y}{(\log N)^{4\upsilon+6\gamma+10}}.
\end{equation}
Thus, we have
\begin{equation*}
\begin{split}
\sum_{\substack{f\le 2\sqrt{X+Y}\\ (f,2)=1}}\frac{1}{f}
\sum_{\substack{X<p\le X+Y\\ f^2\mid D_N(p)}}
L\left(1,\chi_{d_{N,f}(p)}\right)\log p
&=Y\sum_{\substack{f\le V,n\le U\\ (f,2)=1}}\frac{1}{fn\varphi(4nf^2)}
	\sum_{\substack{a\in\Z/4n\Z\\ a\equiv 1\pmod 4}}\leg{a}{n}
	\#C_N(a,n,f) \\
&\quad+O \left( \frac{YN^{7/32}}{\sqrt U}+\frac{\sqrt N\log U\log N}{\sqrt V}+U\log(UN)\log V\right)\\
&\quad+O\left(\frac{YV\sqrt{\log U\log N}}{(\log N)^{2\upsilon+3\gamma+5}}\right).
\end{split}
\end{equation*}

\begin{lma}\label{complete K sum}
For any $U,V,\epsilon>0$, we have
\begin{equation*}
\begin{split}
K_0(N)=
\sum_{\substack{f\le V,n\le U\\ (f,2)=1}}\frac{1}{fn\varphi(4nf^2)}
	\sum_{\substack{a\in\Z/4n\Z\\ a\equiv 1\pmod 4}}\leg{a}{n}\#C_N(a,n,f)	
	+O\left(\frac{N^\epsilon}{\sqrt U}+\frac{\log\log N}{V}\right).
\end{split}
\end{equation*}
\end{lma}
We delay the proof of Lemma~\ref{complete K sum} until Section~\ref{proofs of lemmas}.
Applying the lemma and choosing
\begin{equation*}
U=\frac{Y}{(\log N)^{4\upsilon+10\gamma+18}},\quad V=(\log N)^{2\upsilon+2\gamma +4},
\end{equation*}
we have
\begin{equation*}
\begin{split}
\sum_{\substack{f\le 2\sqrt{X+Y}\\ (f,2)=1}}\frac{1}{f}
\sum_{\substack{X<p\le X+Y\\ f^2\mid D_N(p)}}
L\left(1,\chi_{d_{N,f}(p)}\right)\log p
=K_0(N)Y&
+O\left(\frac{Y}{(\log N)^\gamma}\right)
\end{split}
\end{equation*}
provided that $Y\gg\sqrt N/(\log N)^{\upsilon}$.
Note that our choice of $U,V$ satisfies the condition~\eqref{bdh restriction on U and V}.
\end{proof}

\section{Computing the ``almost constant"}

Recall that $C_N(a,n,f)$ was defined by
\begin{equation*}
C_N(a,n,f)=\left\{z\in(\Z/4nf^2\Z)^*: D_N(z)\equiv af^2\pmod{4nf^2}\right\},
\end{equation*}
where $D_N(z)=z^2-2(N+1)z+(N-1)^2$.
The following is the main result of this section.
\begin{prop}\label{product formula}
With $K(N)$ as defined in Theorem~\ref{main thm} and $K_0(N)$ as defined in
Theorem~\ref{thmaverageSV}, we have
\begin{equation*}
\frac{N}{\varphi(N)}K(N)=K_0(N).
\end{equation*}
\end{prop}

\begin{proof}
By the Chinese Remainder Theorem and the definition of $C_N(a,n,f)$,
\begin{equation*}
\#C_N(a,n,f)=
\prod_{\ell\mid 4nf^2}\#C_N^{(\ell)}(a,n,f),
\end{equation*}
where
\begin{equation}\label{crt pieces of C}
C_N^{(\ell)}(a,n,f):=
\{z\in(\Z/\ell^{\nu_\ell(4nf^2)}\Z)^*:
	D_N(z)\equiv af^2\pmod{\ell^{\nu_\ell(4nf^2)}}\}.
\end{equation}
We require the following lemma whose proof we delay until Section~\ref{proofs of lemmas}.

\begin{lma}\label{counting solutions}
Suppose that $N$ and $f$ are odd and that $a\equiv 1\pmod 4$.
Let $\ell$ be any odd prime dividing $nf$, and let
$e=\nu_\ell(4nf^2)=\nu_\ell(nf^2)$.
If $\ell\dnd 4N+af^2$, then
\begin{equation*}
\#C_N^{(\ell)}(a,n,f)=
\begin{cases}
1+\leg{4N+af^2}{\ell}&\text{if }\ell\dnd (N-1)^2-af^2,\\
1&\text{if }\ell\mid (N-1)^2-af^2.
\end{cases}
\end{equation*}
If $\ell\mid 4N+af^2$, then with $s=\nu_\ell(4N+af^2)$, we have
\begin{equation*}
\begin{split}
\#C_N^{(\ell)}(a,n,f)
&=\begin{cases}
2\leg{N+1}{\ell}^2\ell^{s/2}
	&\text{if }1\le s<e,\ 2\mid s, \text{ and }\leg{(4N+af^2)/\ell^s}{\ell}=1,\\
\leg{N+1}{\ell}^2\ell^{\floor{e/2}}
	&\text{if }s\ge e,\\
0&\text{otherwise}.
\end{cases}
\end{split}
\end{equation*}
In particular, if $\ell\mid f$, then
\begin{equation*}
\#C_N^{(\ell)}(1,1,f)=\begin{cases}
1+\leg{N(N-1)^2}{\ell}&\text{if }\ell\dnd N,\\
2\ell^{\nu_\ell(N)/2}&\text{if }1\le\nu_\ell(N)<2\nu_\ell(f),\ 2\mid\nu_\ell(N),
	\text{ and }\leg{N_\ell}{\ell}=1,\\
\ell^{\nu_\ell(f)}&\text{if }2\nu_\ell(f)\le\nu_\ell(N),\\
0&\text{otherwise},
\end{cases}
\end{equation*}
where $N_\ell=N/\ell^{\nu_\ell(N)}$ is the $\ell$-free part of $N$.
Furthermore,
\begin{equation*}
\#C_N^{(2)}(a,n,f)
=\begin{cases}
2&\text{if }\nu_2(4nf^2)=2+\nu_2(n)=2,\\
4&\text{if }\nu_2(4nf^2)=2+\nu_2(n)\ge 3\text{ and }a\equiv 5\pmod 8,\\
0&\text{otherwise}.
\end{cases}
\end{equation*}
\end{lma}

By Lemma~\ref{counting solutions},
we may write
\begin{equation*}
\#C_N^{(2)}(a,n,f)
=2\mathcal S_2(n,a),
\end{equation*}
where
\begin{equation}\label{defn of S_2}
\mathcal S_2(n,a):=\begin{cases}
1&\text{if } 2\dnd n,\\
2&\text{if }2\mid n\text{ and } a\equiv 5\pmod 8,\\
0&\text{otherwise}.
\end{cases}
\end{equation}

Note that if $\ell$ is a prime dividing $f$ and not dividing $n$, then
$af^2\equiv 0\pmod{\ell^{\nu_\ell(4nf^2)}}$ as $\nu_\ell(4nf^2)=\nu_\ell(f^2)$.
Hence, in this case, $\#C_N^{(\ell)}(a,n,f)$ does not depend on the value of $a$, and so we write
$\#C_N^{(\ell)}(a,n,f)=\#C_N^{(\ell)}(1,1,f)$ if $\ell\mid f$ and $\ell\dnd n$.
Therefore, letting $n'$ denote the odd part of $n$ and
\begin{equation*}
c_{N,f}(n):=\sum_{\substack{a\in(\Z/4n\Z)^*\\ a\equiv 1\pmod 4}}
\leg{a}{n}
\mathcal S_2(n,a)
\prod_{\ell\mid n'}
\#C_N^{(\ell)}(a,n,f),
\end{equation*}
we may write
\begin{equation}\label{separate n and f}
\begin{split}
K_0(N)
&=\sum_{\substack{f=1\\ (f,2)=1}}^\infty\frac{1}{f}\sum_{n=1}^\infty\frac{1}{n\varphi(4nf^2)}
	\sum_{\substack{a\in\Z/4n\Z\\ a\equiv 1\pmod 4}}\leg{a}{n}
	\#C_N(a,n,f)\\
&=\sum_{\substack{f=1\\ (f,2)=1}}^\infty\frac{1}{f^2\varphi(f)}
	\sum_{n=1}^\infty\frac{2\varphi((n,f))}{(n,f)n\varphi(4n)}
	\left[\prod_{\substack{\ell\mid f\\ \ell\dnd n}}\#C_N^{(\ell)}(1,1,f)\right]c_{N,f}(n)\\
&=\sum_{f=1}^\infty\strut^\prime
	\frac{\prod_{\ell\mid f}\#C_N^{(\ell)}(1,1,f)}{f^2\varphi(f)}
	\sum_{n=1}^\infty\frac{2\varphi((n,f))}{(n,f)n\varphi(4n)}
	\left[\prod_{\ell\mid (f,n)}\#C_N^{(\ell)}(1,1,f)\right]^{-1}c_{N,f}(n),
\end{split}
\end{equation}
where the prime on the outer sum indicates that the sum is to be restricted to those $f$ which are
odd and are not divisible by any prime for which $\#C_N^{(\ell)}(1,1,f)=0$.

In order to proceed further, we must show how to compute the function $c_{N,f}(n)$.
We summarize the computation in the following lemma whose proof we also delay
until Section~\ref{proofs of lemmas}.

\begin{lma}\label{compute factors of c}
Suppose that $N$ and $f$ are odd.
The function $c_{N,f}(n)$ is multiplicative in $n$.
Let $\alpha$ be a positive integer and $\ell$ an odd prime.
Then
\begin{equation*}
\frac{c_{N,f}(2^\alpha)}{2^{\alpha-1}}=(-1)^\alpha 2.
\end{equation*}
If $\ell\mid f$ and $\ell\dnd N$, then
\begin{equation*}
\frac{c_{N,f}(\ell^\alpha)}{\ell^{\alpha-1}}
=\#C_N^{(\ell)}(1,1,f)
\begin{cases}
\ell-1&\text{if }2\mid\alpha,\\
0&\text{if }2\dnd\alpha.
\end{cases}
\end{equation*}
If $\ell\mid N$ and $\ell\dnd f$, then
\begin{equation*}
\frac{c_{N,f}(\ell^\alpha)}{\ell^{\alpha-1}}=\ell-2.
\end{equation*}
If $\ell\dnd Nf$, then
\begin{equation*}
\begin{split}
\frac{c_{N,f}(\ell^\alpha)}{\ell^{\alpha-1}}
&=\begin{cases}
\ell-2-\leg{N}{\ell}-\leg{N^2-1}{\ell}^2
+\leg{N+1}{\ell}^2
	&\text{if }2\mid\alpha,\\
-1-\leg{-N}{\ell}-\leg{N^2-1}{\ell}^2
+\leg{-N(N+1)^2}{\ell}
	&\text{if }2\dnd\alpha
\end{cases}\\
&=\begin{cases}
\ell-1-\leg{N}{\ell}-\leg{N-1}{\ell}^2
	&\text{if }2\mid\alpha,\\
-1-\leg{N-1}{\ell}^2
	&\text{if }2\dnd\alpha.
\end{cases}
\end{split}
\end{equation*}
If $\ell\mid (f,N)$ and $2\nu_\ell(f)<\nu_\ell(N)$, then
\begin{equation*}
\frac{c_{N,f}(\ell^\alpha)}{\ell^{\alpha-1}}
=\#C_N^{(\ell)}(1,1,f)(\ell-1).
\end{equation*}
If $\ell\mid (f,N)$ and $\nu_\ell(N)<2\nu_\ell(f)$, then
\begin{equation*}
\frac{c_{N,f}(\ell^\alpha)}{\ell^{\alpha-1}}
=\#C_N^{(\ell)}(1,1,f)\begin{cases}
\ell-1&\text{if }2\mid \alpha,\\
0&\text{if }2\dnd\alpha.
\end{cases}
\end{equation*}
If $\ell\mid (f,N)$ and $\nu_\ell(N)=2\nu_\ell(f)$, then
\begin{equation*}
\frac{c_{N,f}(\ell^\alpha)}{\ell^{\alpha-1}}
=\#C_N^{(\ell)}(1,1,f)\begin{cases}
\left(\ell-1-\leg{N_\ell}{\ell}+\leg{-N_\ell}{\ell}\right)&\text{if }2\mid\alpha,\\
\left(\leg{-N_\ell}{\ell}-1\right)&\text{if }2\dnd\alpha,\\
\end{cases}
\end{equation*}
where $N_\ell=N/\ell^{\nu_\ell(N)}$ denotes the $\ell$-free part of $N$.
Furthermore, for any $n$, we have the bound
\begin{equation*}
c_{N,f}(n)\ll\frac{n\prod_{\ell\mid (f,n)}\#C_N(1,1,f)}{\kappa_{2N}(n)},
\end{equation*}
where for any integer $m$, $\kappa_m(n)$ is the multiplicative function defined on prime
powers by
\begin{equation}\label{defn of kappa}
\kappa_m(\ell^\alpha)
:=\begin{cases}
\ell&\text{if }2\dnd\alpha\text{ and }\ell\dnd m,\\
1&\text{otherwise}.
\end{cases}
\end{equation}
\end{lma}

Recalling the restrictions on $f$ in~\eqref{separate n and f} and applying
Lemma~\ref{compute factors of c}, the sum over $n$ in~\eqref{separate n and f} may be factored as
\begin{equation*}
\begin{split}
&\sum_{n=1}^\infty\frac{2\varphi((n,f))}{(n,f)n\varphi(4n)}
	\left[\prod_{\ell\mid (f,n)}\#C_N^{(\ell)}(1,1,f)\right]^{-1}c_{N,f}(n)\\
&\quad=\left\{\sum_{\alpha\ge 0}\frac{2c_{N,f}(2^\alpha)}{2^\alpha\varphi(2^{\alpha+2})}\right\}
	\prod_{\substack{\ell\dnd f\\ \ell\ne 2}}
		\left\{\sum_{\alpha\ge 0}\frac{c_{N,f}(\ell^\alpha)}{\ell^\alpha\varphi(\ell^\alpha)}\right\}
	\prod_{\ell\mid f}
		 \left\{1+ \sum_{\alpha\ge 1}\frac{\varphi((\ell^\alpha,f))c_{N,f}(\ell^\alpha)}
 		 {(\ell^\alpha,f)\ell^\alpha\varphi(\ell^\alpha)\#C_N^{(\ell)}(1,1,f)}\right\}
\\
 &\quad
 =\frac{2}{3}
	\prod_{\substack{\ell\dnd f\\ \ell\mid N}}F_0(\ell)	
	\prod_{\substack{\ell\dnd f\\ \ell\dnd N\\ \ell\ne 2}}F_1(\ell)
	\prod_{\ell\mid f}F_2(\ell,f),
\end{split}
\end{equation*}
where for any odd prime $\ell$, we make the definitions
\begin{align*}
F_0(\ell)&:=\left(1+\frac{\ell-2}{(\ell-1)^2}\right),\\
F_1(\ell)
&:=\left(1-
	\frac{\leg{N-1}{\ell}^2\ell
	+\leg{N}{\ell}+\leg{N-1}{\ell}^2+1}{(\ell-1)(\ell^2-1)}
\right),\\
F_2(\ell,f)&:=\begin{cases}
\left(1+\frac{1}{\ell(\ell+1)}\right)&\text{if }\nu_\ell(N)<2\nu_\ell(f),\\
\left(1+\frac{1}{\ell}\right)&\text{if }\nu_\ell(N)>2\nu_\ell(f),\\
\left(1+
	\frac{\leg{-N_\ell}{\ell}\ell+\leg{-N_\ell}{\ell}-\leg{N_\ell}{\ell}-1}{\ell(\ell^2-1)}
	\right)
	&\text{if }\nu_\ell(N)=2\nu_\ell(f).
\end{cases}
\end{align*}
Substituting this back into equation~\eqref{separate n and f}, we have
\begin{equation}\label{f sum remaining}
K_0(N)
=\frac{2}{3}
	\prod_{\ell\mid N}F_0(\ell)
	\prod_{\substack{\ell \dnd N\\ \ell\ne 2}} F_1(\ell)
	\sum_{\substack{f=1\\ (f,2)=1}}^{\infty}\strut^\prime
		\frac{\prod_{\ell\mid f}\#C_N^{(\ell)}(1,1,f)}{\varphi(f)f^2}
		\prod_{\substack{\ell\mid f\\ \ell\mid N}}\frac{F_2(\ell,f)}{F_0(\ell)}
		\prod_{\substack{\ell\mid f\\ \ell\dnd N}}\frac{F_2(\ell,f)}{F_1(\ell)}.
\end{equation}
The sum over $f$ may be factored as
\begin{equation*}
\begin{split}
&\sum_{\substack{f=1\\ (f,2)=1}}^{\infty}\strut^\prime
		\frac{\prod_{\ell\mid f}\#C_N^{(\ell)}(1,1,f)}{\varphi(f)f^2}
		\prod_{\substack{\ell\mid f\\ \ell\mid N}}\frac{F_2(\ell,f)}{F_0(\ell)}
		\prod_{\substack{\ell\mid f\\ \ell\dnd N}}\frac{F_2(\ell,f)}{F_1(\ell)}\\
&\quad=\prod_{\ell\mid N}
	\left\{1+\sum_{\alpha\ge 1}\frac{\#C_N^{(\ell)}(1,1,\ell^\alpha) F_2(\ell,\ell^\alpha)}
		{\varphi(\ell^\alpha)\ell^{2\alpha}F_0(\ell)}
		\right\}
\prod_{\substack{\ell\dnd N\\ \ell\ne 2}}
	\left\{1+\sum_{\alpha\ge 1}\frac{\#C_N^{(\ell)}(1,1,\ell^\alpha) F_2(\ell,\ell^\alpha)}
		{\varphi(\ell^\alpha)\ell^{2\alpha}F_1(\ell)}
	\right\}.
\end{split}
\end{equation*}
When $\ell\dnd 2N$, the factor simplifies as
\begin{equation*}
\begin{split}
1+\sum_{\alpha\ge 1}\frac{\#C_N^{(\ell)}(1,1,\ell^\alpha) F_2(\ell,\ell^\alpha)}
		{\varphi(\ell^\alpha)\ell^{2\alpha}F_1(\ell)}
&=1+\frac{\ell C_N^{(\ell)}(1,1,\ell)F_2(\ell,\ell)}{F_1(\ell)(\ell-1)}
	\sum_{\alpha\ge 1}\frac{1}{\ell^{3\alpha}}\\
&=1+\frac{\left(1+\leg{N(N-1)^2}{\ell}\right)(\ell^2+\ell+1)}{(\ell^2-1)(\ell^3-1)F_1(\ell)}\\
&=1+\frac{1+\leg{N(N-1)^2}{\ell}}{(\ell^2-1)(\ell-1)F_1(\ell)}.
\end{split}
\end{equation*}
When $\nu_\ell(N)$ is odd, the factor 
simplifies as
\begin{equation*}
\begin{split}
1&+\sum_{\alpha\ge 1}\frac{\#C_N^{(\ell)}(1,1,\ell^\alpha) F_2(\ell,\ell^\alpha)}
		{\varphi(\ell^\alpha)\ell^{2\alpha}F_0(\ell)}\\
\quad&=1+\frac{\ell}{F_0(\ell)(\ell-1)}\sum_{\alpha\ge 1}
	\frac{\#C_N^{(\ell)}(1,1,\ell^\alpha) F_2(\ell,\ell^\alpha)}{\ell^{3\alpha}}\\
&=1+\frac{\ell}{F_0(\ell)(\ell-1)}
	\sum_{\alpha=1}^{\floor{\nu_\ell(N)/2}}
	\frac{\ell^{\alpha}\left(1+\frac{1}{\ell}\right)}{\ell^{3\alpha}}\\
&=1+\frac{(\ell+1)(1-\ell^{1-\nu_\ell(N)})}{F_0(\ell)(\ell-1)(\ell^2-1)}\\
&=1+\frac{1-\ell^{1-\nu_\ell(N)}}{F_0(\ell)(\ell-1)^2}\\
&=1+\frac{\ell^{\nu_\ell(N)}-\ell}{F_0(\ell)\ell^{\nu_\ell(N)}(\ell-1)^2}.
\end{split}
\end{equation*}
When $\nu_\ell(N)$ positive, even, and $\leg{N_\ell}{\ell}=-1$, the factor simplifies as
\begin{equation*}
\begin{split}
1&+\sum_{\alpha\ge 1}\frac{\#C_N^{(\ell)}(1,1,\ell^\alpha) F_2(\ell,\ell^\alpha)}
		{\varphi(\ell^\alpha)\ell^{2\alpha}F_0(\ell)}\\
\quad&=1+\frac{\ell}{F_0(\ell)(\ell-1)}\sum_{\alpha\ge 1}
	\frac{\#C_N^{(\ell)}(1,1,\ell^\alpha) F_2(\ell,\ell^\alpha)}{\ell^{3\alpha}}\\
&=1+\frac{\ell^{\nu_\ell(N)}-\ell^2}{F_0(\ell)\ell^{\nu_\ell(N)}(\ell-1)^2}
	+\frac{\ell\#C_N^{(\ell)}(1,1,\ell^{\nu_\ell(N)/2})F_2(\ell,\ell^{\nu_\ell(N)/2})}
	{F_0(\ell)(\ell-1)\ell^{3\nu_\ell(N)/2}}\\
&=1+\frac{\ell^{\nu_\ell(N)}-\ell^2}{F_0(\ell)\ell^{\nu_\ell(N)}(\ell-1)^2}
	+\frac{\ell^2-\ell-\leg{-1}{\ell}}{F_0(\ell)\ell^{\nu_\ell(N)}(\ell-1)^2}\\
&=1+\frac{\ell^{\nu_\ell(N)}-\ell-\leg{-1}{\ell}}{F_0(\ell)\ell^{\nu_\ell(N)}(\ell-1)^2}\\
&=1+\frac{\ell^{\nu_\ell(N)}-\ell+\leg{-N_\ell}{\ell}}{F_0(\ell)\ell^{\nu_\ell(N)}(\ell-1)^2}.
\end{split}
\end{equation*}
When $\nu_\ell(N)$ positive, even, and $\leg{N_\ell}{\ell}=1$, the factor simplifies as
\begin{equation*}
\begin{split}
1&+\sum_{\alpha\ge 1}\frac{\#C_N^{(\ell)}(1,1,\ell^\alpha) F_2(\ell,\ell^\alpha)}
		{\varphi(\ell^\alpha)\ell^{2\alpha}F_0(\ell)}
=1+\frac{\ell}{F_0(\ell)(\ell-1)}\sum_{\alpha\ge 1}
	\frac{\#C_N^{(\ell)}(1,1,\ell^\alpha) F_2(\ell,\ell^\alpha)}{\ell^{3\alpha}}\\
&=1+\frac{\ell^{\nu_\ell(N)}-\ell}{F_0(\ell)\ell^{\nu_\ell(N)}(\ell-1)^2}
	+\frac{\ell\#C_N^{(\ell)}(1,1,\ell^{\nu_\ell(N)/2})F_2(\ell,\ell^{\nu_\ell(N)/2})}
	{F_0(\ell)(\ell-1)\ell^{3\nu_\ell(N)/2}}\\
&\quad+\frac{\ell}{F_0(\ell)(\ell-1)}\sum_{\alpha=\frac{\nu_\ell(N)}{2}+1}^\infty
	\frac{\#C_N^{(\ell)}(1,1,\ell^\alpha) F_2(\ell,\ell^\alpha)}{\ell^{3\alpha}}\\
&=1+\frac{\ell^{\nu_\ell(N)}-\ell^2}{F_0(\ell)\ell^{\nu_\ell(N)}(\ell-1)^2}
	+\frac{\ell(\ell^2-1)+\leg{-1}{\ell}\ell+\leg{-1}{\ell}-2}
	{F_0(\ell)(\ell-1)(\ell^2-1)\ell^{\nu_\ell(N)}}\\
&\quad+\frac{\ell}{F_0(\ell)(\ell-1)}\sum_{\alpha=\frac{\nu_\ell(N)}{2}+1}^\infty
	\frac{2\ell^{\nu_\ell(N)/2}\left(1+\frac{1}{\ell(\ell+1)}\right)}{\ell^{3\alpha}}\\
&=1+\frac{\ell^{\nu_\ell(N)}-\ell^2}{F_0(\ell)\ell^{\nu_\ell(N)}(\ell-1)^2}
	+\frac{\ell^2-\ell+\leg{-1}{\ell}}
	{F_0(\ell)(\ell-1)^2\ell^{\nu_\ell(N)}}\\
&=1+\frac{\ell^{\nu_\ell(N)}-\ell+\leg{-N_\ell}{\ell}}{F_0(\ell)(\ell-1)^2\ell^{\nu_\ell(N)}}.
\end{split}
\end{equation*}

Substituting this back into~\eqref{f sum remaining}, we find that
\begin{equation*}
\begin{split}
K_0(N)
&=\frac{2}{3}
	\prod_{\ell \dnd 2N}\left(F_1(\ell)+\frac{1+\leg{N(N-1)^2}{\ell}}{(\ell^2-1)(\ell-1)}\right)
	\prod_{\substack{\ell\mid N\\ 2\dnd\nu_\ell(N)}}
	\left(F_0(\ell)+\frac{(\ell^{\nu_\ell(N)}-\ell)}{\ell^{\nu_\ell(N)}(\ell-1)^2}\right)\\
&\quad\times
	\prod_{\substack{\ell\mid N\\ 2\mid\nu_\ell(N)}}
	\left(F_0(\ell)
	+\frac{\ell^{\nu_\ell(N)}-\ell+\leg{-N_\ell}{\ell}}{\ell^{\nu_\ell(N)}(\ell-1)^2}\right)\\
&=\frac{2}{3}
	\prod_{\ell \dnd 2N}\left(
	1-\frac{\leg{N-1}{\ell}^2\left[\ell+1-\leg{N}{\ell}\right]+\leg{N}{\ell}}{(\ell-1)(\ell^2-1)}
	\right)
	\prod_{\substack{\ell\mid N\\ 2\dnd\nu_\ell(N)}}
	\left(1+\frac{\ell^{\nu_\ell(N)+1}-\ell^{\nu_\ell(N)}-\ell}
	{\ell^{\nu_\ell(N)}(\ell-1)^2}\right)\\
&\quad\times
	\prod_{\substack{\ell\mid N\\ 2\mid\nu_\ell(N)}}
	\left(1+\frac{\ell^{\nu_\ell(N)+1}-\ell^{\nu_\ell(N)}-\ell+\leg{-N_\ell}{\ell}}
	{\ell^{\nu_\ell(N)}(\ell-1)^2}\right)\\
&=\frac{N}{\varphi(N)}
	\prod_{\ell \dnd N}\left(
	1-\frac{\leg{N-1}{\ell}^2\ell+1}{(\ell-1)(\ell^2-1)}
	\right)
	\prod_{\substack{\ell\mid N\\ 2\dnd\nu_\ell(N)}}
	\left(1-\frac{1}{\ell^{\nu_\ell(N)}(\ell-1)}\right)\\
&\quad\times
	\prod_{\substack{\ell\mid N\\ 2\mid\nu_\ell(N)}}
	\left(1-\frac{\ell-\leg{-N_\ell}{\ell}}
	{\ell^{\nu_\ell(N)+1}(\ell-1)}\right).
\end{split}
\end{equation*}
\end{proof}

\section{Proof of Theorem~\ref{main thm}}\label{proof of main thm}

We are now ready to give the proof of our main result.
\begin{proof}[Proof of Theorem~\ref{main thm}.]
By Proposition~\ref{average order in terms of class numbers},
we see that Theorem~\ref{main thm} follows if we show that
\begin{equation*}
\sum_{N^-<p< N^+}\frac{H(D_N(p))}{p}
=K(N)\frac{N}{\varphi(N)\log N}+O\left(\frac{1}{(\log N)^{1+\gamma}}\right).
\end{equation*}
We begin by dividing the interval $(N^-,N^+)$ into intervals of length
$Y:=\frac{\sqrt N}{\floor{(\log N)^{\gamma+2}}}$.  For each integer $k$ in
$I:=[-2\sqrt N/Y, 2\sqrt N/Y)\cap\Z$,
we write $X=X_k:=N+1+kY$.
Thus,
\begin{equation}\label{break into shorter averages}
\sum_{N^-<p< N^+}\frac{H(D_N(p))}{p}
=\sum_{k\in I}\sum_{X_k<p\le X_k+Y}\frac{H(D_N(p))}{p}.
\end{equation}

Recalling the definition of the Kronecker class number, the definition of $d_{N,f}(p)$, and the
class number formula
(see~\eqref{defn of K-class no},~\eqref{defn of d}, and~\eqref{class no formula}),
we have the identity
\begin{equation}\label{analytic Kronecker class number formula}
H(D_N(p))=\frac{1}{2\pi}\sum_{\substack{f^2|D_N(p)\\ d_{N,f}(p)\equiv 0,1\pmod 4}}
\sqrt{|d_{N,f}(p)|}L\left(1,\chi_{d_{N,f}(p)}\right).
\end{equation}
We assume that $N$ is large enough so that there are only odd primes $p$ satisfying
the condition $N^-<p<N^+$.
Therefore, since we assumed that $N$ is odd,
\begin{equation*}
D_N(p)=(p+1-N)^2-4p\equiv 1\pmod 4,
\end{equation*}
and it follows that there are only odd $f$ in the above sum and that
$d_{N,f}(p)\equiv 1\pmod 4$ for each such $f$.
Hence, summing~\eqref{analytic Kronecker class number formula} over all primes $p$
in the range $X<p\le X+Y$ and switching the order of summation, we have
\begin{equation}\label{reduce to short average of L-values}
\begin{split}
\sum_{X<p\le X+Y}\frac{H(D_N(p))}{p}
&=\frac{1}{2\pi}\sum_{X<p\le X+Y}\frac{1}{p}
\sum_{f^2|D_N(p)}
\sqrt{|d_{N,f}(p)|}L\left(1,\chi_{d_{N,f}(p)}\right)\\
&=\frac{1}{2\pi}\sum_{\substack{f\le 2\sqrt{X+Y}\\ (f,2)=1}}\frac{1}{f}
\sum_{\substack{X<p\le X+Y\\ f^2\mid D_N(p)}}
\frac{\sqrt{|D_N(p)|}}{p}L\left(1,\chi_{d_{N,f}(p)}\right).
\end{split}
\end{equation}

We now change ``weights," approximating $\frac{\sqrt{|D_N(p)|}}{p}$ by
$\frac{\sqrt{|D_N(X)|}\log p}{N\log N}$.  If $p$ is a prime in the interval $(X,X+Y]$,
then $p=X+O(Y)$, and hence
\begin{equation*}
D_N(p)
=D_N(X)+O(Y\sqrt N).
\end{equation*}
Let $X^*$ be the value minimizing the function $\sqrt{|D_N(t)|}$ on the interval $[X,X+Y]$.
Since it is also true that $p=N+O(\sqrt N)$, we have that
\begin{equation*}
\left|\frac{\sqrt{|D_N(p)|}}{p}-\frac{\sqrt{|D_N(X)|}\log p}{N\log N}\right|
\ll\begin{cases}
\frac{\sqrt{|D_N(X)|}}{N^{3/2}}+\frac{\sqrt{Y}}{N^{3/4}}&\text{if }N^{\pm}\in [X,X+Y],\\
\frac{\sqrt{|D_N(X)|}}{N^{3/2}}+\frac{Y}{N^{1/2}\sqrt{|D_N(X^*)|}}&\text{otherwise.}
\end{cases}
\end{equation*}
Hence, by Theorem~\ref{thmaverageSV} and Proposition~\ref{product formula}, the right hand side
of~\eqref{reduce to short average of L-values} is equal to
\begin{equation}\label{change weights}
\begin{split}
\frac{K(N)Y\sqrt{|D_N(X)|}}{2\pi \varphi(N)\log N}
+\begin{cases}
O\left(\frac{Y\sqrt{|D_N(X)|}}{N(\log N)^{\gamma+1}}
	+\frac{Y\sqrt{|D_N(X)|}\log N}{N^{3/2}}+\frac{Y^{3/2}\log N}{N^{3/4}}\right)
	&\text{ if }N^\pm\in [X,X+Y],\\
O\left(
\frac{Y\sqrt{|D_N(X)|}}{N(\log N)^{\gamma+1}}
	+\frac{Y\sqrt{|D_N(X)|}\log N}{N^{3/2}}
	+\frac{Y^2\log N}{N^{1/2}\sqrt{|D_N(X^*)|}} \right)
	&\text{otherwise}.
\end{cases}
\end{split}
\end{equation}
Since $D_N(X_k)=0$ for $k$ on the endpoints of the interval $[-2\sqrt N/Y,2\sqrt N/Y]\supset I$,
by the Euler-Maclaurin summation formula, we have
\begin{equation*}
\begin{split}
\sum_{k\in I}\sqrt{|D_N(X_k)|}
&=\int_{-2\sqrt N/Y}^{2\sqrt N/Y}\sqrt{4N-(tY)^2}\diff t
+O\left(\int_{-2\sqrt N/Y}^{2\sqrt N/Y}\frac{\left| tY^2\right|}{\sqrt{4N-(tY)^2}}\diff t\right)\\
&=\frac{2\pi N}{Y}
+O\left(\sqrt N\right),
\end{split}
\end{equation*}
and furthermore,
\begin{equation*}
\sum_{k\in I}\frac{1}{\sqrt{|D_N(X_k^*)|}}
\ll\int_{-2\sqrt N/Y}^{2\sqrt N/Y}\frac{\diff t}{\sqrt{4N-(tY)^2}}
=\frac{1}{Y}\int_{-2\sqrt N}^{2\sqrt N}\frac{\diff u}{\sqrt{4N-u^2}}
=\frac{\pi}{Y}.
\end{equation*}
Whence, summing~\eqref{change weights} over $k\in I$, we have
\begin{equation*}
\sum_{N^-<p< N^+}\frac{H(D_N(p))}{p}
=K(N)\frac{N}{\varphi(N)\log N}
	+O\left( \frac{1}{(\log N)^{\gamma+1}}
	+\frac{Y\log N}{\sqrt N}
	\right).
\end{equation*}
\end{proof}

\section{Proofs of lemmas}\label{proofs of lemmas}


In this section, we give the proofs of the technical lemmas needed in the rest of the paper.

\begin{lma}\label{bound quad cong soln count}
For every positive integer $f$,
\begin{equation*}
\#\{a\in\Z/f\Z: D_N(a)\equiv 0\pmod f\}\ll\sqrt f.
\end{equation*}
\end{lma}
\begin{rmk}
Recall that since this lemma was used in the proof of Proposition~\ref{upper bound for H on avg},
we do not assume that $N$ is odd here.
\end{rmk}
\begin{proof}[Proof of Lemma~\ref{bound quad cong soln count}.]
First, we use the Chinese Remainder Theorem to write
\begin{equation*}
\#\{a\in\Z/f\Z: D_N(a)\equiv 0\pmod{f}\}
=\prod_{\ell\mid f}\#\{a\in\Z/\ell^{\nu_\ell(f)}\Z: D_N(a)\equiv 0\pmod{\ell^{\nu_\ell(f)}}\}.
\end{equation*}
We will show that
\begin{equation}\label{prime power mod bound}
\#\{a\in\Z/\ell^{e}\Z: D_N(a)\equiv 0\pmod{\ell^{e}}\}
\le\begin{cases}
\max\{\ell^{\floor{e/2}},2\ell^{(e-1)/2}\}&\text{if }\ell>2,\\
\max\{\ell^{\floor{e/2}},4\ell^{(e-1)/2}\}&\text{if }\ell=2.
\end{cases}
\end{equation}
From this, we readily deduce that
\begin{equation*}
\#\{a\in\Z/f\Z: D_N(a)\equiv 0\pmod{f}\}\le 8\sqrt f,
\end{equation*}
which is a more precise result than stated in the lemma.

We now give the proof of~\eqref{prime power mod bound}.
Since
\begin{equation*}
D_N(a)=a^2-2(N+1)a+(N-1)^2=(a-N-1)^2-4N,
\end{equation*}
it suffices to consider the number of solutions to the congruence
\begin{equation}\label{change of var}
Z^2\equiv 4N\pmod{\ell^e}.
\end{equation}
Suppose $z$ is an integer solution to~\eqref{change of var} and
write $4N=\ell^sN_0$ with $(\ell,N_0)=1$.
If $s\ge e$, it follows that $z\equiv 0\pmod{\ell^{\ceil{e/2}}}$, and hence there are at most
$\ell^{\floor{e/2}}$ solutions to~\eqref{change of var}.
Thus, we may assume that $s<e$ and write
\begin{equation*}
z^2= \ell^s(N_0+\ell^{e-s}k)
\end{equation*}
for some integer $k$.
Since $(\ell,N_0)=1$, it follows that $s$ must be even.
Writing $s=2s_0$, we see that $z=\ell^{s_0}x$, where
$x$ is an integer solution to the congruence
\begin{equation}\label{reduced cong}
X^2\equiv N_0\pmod{\ell^{e-s}}.
\end{equation}
So, in particular, $z\equiv\ell^{s_0}x\pmod{\ell^{e-s_0}}$.
Since $(\ell,N_0)=1$, it is a classical exercise as in~\cite[p.~98]{HW:1979} to show that
\begin{equation*}
\#\{X\in\Z/\ell^{e-s}\Z: X^2\equiv N_0\pmod{\ell^{e-s}}\}
\le\begin{cases}
2&\text{if }\ell>2,\\
4&\text{if }\ell=2.
\end{cases}
\end{equation*}
Therefore, there are at most $2\ell^{e-(e-s_0)}=2\ell^{s/2}\le 2\ell^{(e-1)/2}$ solutions
to~\eqref{change of var} when $\ell$ is odd, and there are at most
$4\ell^{e-(e-s_0)}=4\ell^{s/2}\le 4\ell^{(e-1)/2}$
solutions when $\ell=2$.
\end{proof}

\begin{proof}[Proof of Lemma~\ref{complete K sum}]
By Lemma~\ref{compute factors of c},
$c_{N,f}(n)\ll \frac{n\prod_{\ell\mid(f,n)}\#C_N^{(\ell)}(1,1,f)}{\kappa_{2N}(n)}$,
where for any positive integer $m$,
$\kappa_{m}(n)$ is the multiplicative function defined on the prime powers
by~\eqref{defn of kappa}.
Therefore,
\begin{equation}\label{diff bn K and truncated K}
\begin{split}
K_0(N)&-\sum_{\substack{f\le V\\ (f,2)=1}}\frac{1}{f}\sum_{n\le U}\frac{1}{n\varphi(4nf^2)}
	\sum_{\substack{a\in\Z/4n\Z\\ a\equiv 1\pmod 4}}\leg{a}{n}\#C_N(a,n,f)\\
&\quad\ll\sum_{f\le V}\strut^\prime\frac{\prod_{\ell\mid f}\#C_N^{(\ell)}(1,1,f)}{f^2\varphi(f)}
	\sum_{n>U}\frac{2\varphi((n,f))c_{N,f}(n)}{(n,f)n\varphi(4n)
	\prod_{\ell\mid (n,f)}\#C_N^{(\ell)}(1,1,f)}\\
	&\quad\quad+\sum_{f> V}\strut^\prime\frac{\prod_{\ell\mid f}\#C_N^{(\ell)}(1,1,f)}{f^2\varphi(f)}
	\sum_{n\ge 1}\frac{2\varphi((n,f))c_{N,f}(n)}{(n,f)n\varphi(4n)
	\prod_{\ell\mid (n,f)}\#C_N^{(\ell)}(1,1,f)}\\
&\quad\ll\sum_{\substack{f\le V\\ (f,2)=1}}\frac{\prod_{\ell\mid f}\#C_N^{(\ell)}(1,1,f)}{f^2\varphi(f)}
	\sum_{n>U}\frac{1}{\kappa_{2N}(n)\varphi(n)}\\
	&\quad\quad+\sum_{\substack{f> V\\ (f,2)=1}}
	\frac{\prod_{\ell\mid f}\#C_N^{(\ell)}(1,1,f)}{f^2\varphi(f)}
	\sum_{n\ge 1}\frac{1}{\kappa_{2N}(n)\varphi(n)},
\end{split}
\end{equation}
where the primes on the sums on $f$ are meant to indicate that the sums are to be restricted
to odd $f$ such that $\#C_N^{(\ell)}(1,1,f)\ne 0$ for all primes $\ell$ dividing $f$.

In~\cite[Lemma 3.4]{DP:1999}, we find that
\begin{equation*}
\sum_{n>U}\frac{1}{\kappa_1(n)\varphi(n)}
\sim\frac{c_0}{\sqrt U}
\end{equation*}
for some positive constant $c_0$.
In particular, this implies that the full sum converges.
From this we obtain a crude bound for the tail of the sum over $n$
\begin{equation*}
\begin{split}
\sum_{n>U}\frac{1}{\kappa_{2N}(n)\varphi(n)}
&=\sum_{\substack{km>U\\ (m,2N)=1\\ \ell\mid k\Rightarrow\ell\mid 2N}}
	\frac{1}{\kappa_1(m)\varphi(m)\varphi(k)}
=\sum_{\substack{k\ge 1\\ \ell\mid k\Rightarrow\ell\mid 2N}}\frac{1}{\varphi(k)}
	\sum_{\substack{m>U/k\\ (m,2N)=1}}\frac{1}{\kappa_1(m)\varphi(m)}\\
&\ll\sum_{\substack{k\ge 1\\ \ell\mid k\Rightarrow\ell\mid 2N}}\frac{1}{\varphi(k)}
	\frac{\sqrt k}{\sqrt U}
\ll\frac{1}{\sqrt U}\prod_{\ell\mid N}\left(1+\frac{\ell}{(\ell-1)(\sqrt\ell-1)}\right)\\
&=\frac{1}{\sqrt U}\frac{N}{\varphi(N)}
	\prod_{\ell\mid N}\left(1+\frac{1}{\sqrt\ell(\ell-1)}\right)\left(1+\frac{1}{\sqrt\ell}\right)\\
&\ll\frac{1}{\sqrt U}\frac{N}{\varphi(N)}
	\prod_{\ell\mid N}\left(1+\frac{1}{\sqrt\ell}\right).
\end{split}
\end{equation*}
We have already noted that $N/\varphi(N)\ll\log\log N$.
It is a straightforward exercise as in~\cite[p.~63]{MV:2007} to show that
\begin{equation*}
\prod_{\ell\mid N}\left(1+\frac{1}{\sqrt\ell}\right)
<\exp\left\{O\left(\frac{\sqrt{\log N}}{\log\log N}\right)\right\}.
\end{equation*}
Thus, we conclude that
\begin{equation}\label{n tail}
\sum_{n>U}\frac{1}{\kappa_{2N}(n)\varphi(n)}
\ll\frac{N^\epsilon}{\sqrt U}
\end{equation}
for any $\epsilon>0$.
For the full sum over $n$, we need a sharper bound in the $N$-aspect, which we obtain
by writing
\begin{equation}\label{full n sum bound}
\begin{split}
\sum_{n\ge 1}\frac{1}{\kappa_{2N}(n)\varphi(n)}
&=\prod_{\ell\mid 2N}\left(1+\frac{\ell}{(\ell-1)^2}\right)
	\sum_{\substack{n\ge 1\\ (n,2N)=1}}\frac{1}{\kappa_{2N}(n)\varphi(n)}\\
&\le\frac{2N}{\varphi(2N)}\prod_{\ell\mid 2N}\left(1+\frac{1}{\ell(\ell-1)}\right)
	\sum_{n\ge 1}\frac{1}{\kappa_1(n)\varphi(n)}\\
&\ll\log\log N.	
\end{split}
\end{equation}

For any odd prime $\ell$ dividing $f$, we obtain the bounds
\begin{equation*}
\#C_N^{(\ell)}(1,1,f)
\le\begin{cases}
2\ell^{\nu_\ell(N)/2}&\text{if }\nu_\ell(f)>\nu_\ell(N)/2\text{ and }2\mid\nu_\ell(N),\\
\ell^{\nu_\ell(f)}&\text{otherwise}
\end{cases}
\end{equation*}
from Lemma~\ref{counting solutions}.
However, if $\nu_\ell(f)>\nu_\ell(N)/2$ and $2\mid\nu_\ell(N)$,
it follows that $\nu_\ell(f)\ge 1+\nu_\ell(N)/2$, and hence
\begin{equation*}
2\ell^{\nu_\ell(N)/2}\le\ell^{1+\nu_\ell(N)/2}\le\ell^{\nu_\ell(f)}
\end{equation*}
since $\ell>2$.
Therefore, for every odd integer $f$, we have that
\begin{equation*}
\prod_{\ell\mid f}\#C_N^{(\ell)}(1,1,f)\le f,
\end{equation*}
and hence
\begin{equation}\label{f tail}
\sum_{\substack{f> V\\ (f,2)=1}}\frac{\prod_{\ell\mid f}\#C_N^{(\ell)}(1,1,f)}{f^2\varphi(f)}
<\sum_{f>V}\frac{1}{f\varphi(f)}\ll\frac{1}{V}.
\end{equation}
Substituting the bounds~\eqref{n tail},~\eqref{full n sum bound},
and~\eqref{f tail} into~\eqref{diff bn K and truncated K}, the lemma follows.
\end{proof}

\begin{proof}[Proof of Lemma~\ref{counting solutions}]
Upon completing the square, we have that
\begin{equation*}
z^2-2(N+1)z+(N-1)^2-af^2=(z-N-1)^2-(4N+af^2).
\end{equation*}
Since $N$ is odd, the number of invertible solutions to the congruence
\begin{equation*}
(z-N-1)^2\equiv 4N+af^2\pmod{2^\alpha}
\end{equation*}
is the same as the number of invertible solutions to the congruence
$Z^2\equiv 4N+af^2\pmod{2^\alpha}$.
Since $4N+af^2\equiv 4+a\pmod{8}$, the computation of
$C_N^{(2)}(a,n,f)$ is thus reduced to a standard exercise.
See~\cite[p.~98]{HW:1979} for example.

If $\ell\dnd 4N+af^2$, then there are exactly $1+\leg{4N+af^2}{\ell}$
solutions to the congruence
\begin{equation}\label{odd prime power cong}
(z-N-1)^2\equiv 4N+af^2\pmod{\ell^\alpha}.
\end{equation}
However,
if $\ell$ divides the constant term, $(N-1)^2-af^2$, then we have exactly one
invertible solution and exactly one noninvertible solution.

It remains to treat the case when $\ell\mid 4N+af^2$.
We write $4N+af^2=\ell^sm$ with $(m,\ell)=1$.
First, we observe that any solution $z$ to~\eqref{odd prime power cong} must
satisfy $z\equiv N+1\pmod{\ell}$.  Therefore, if $\ell\mid N+1$, then there
are no invertible solutions; if $\ell\dnd N+1$, then every solution is invertible.
Hence, we assume that $\ell\dnd N+1$.  Now, the number of invertible
solutions to~\eqref{odd prime power cong} is equal to the number of
(noninvertible) solutions to
\begin{equation}\label{odd prime power cong with vanishing disc}
Z^2\equiv \ell^sm\pmod{\ell^e}.
\end{equation}
If $s\ge e$, then $Z$ is a solution if and only if
$Z\equiv 0\pmod{\ell^{\ceil{e/2}}}$.
There are exactly $\ell^{e-\ceil{e/2}}=\ell^{\floor{e/2}}$ such values for $Z$
modulo $\ell^e$.
Now, suppose that $0<s<e$.
Then $Z$ is a solution to~\eqref{odd prime power cong with vanishing disc} if and only if
\begin{equation*}
Z^2=\ell^sm+\ell^e k=\ell^s\left(m+\ell^{e-s}\right)
\end{equation*}
for some integer $k$.
Since $\ell\dnd m$ and $s<e$, we see that there can be no such $Z$ if
$s$ is odd or if $\leg{m}{\ell}=-1$.
Thus, we assume that $s=2s_0$, $\leg{m}{\ell}=1$, and we write
$r^2=m+\ell^{e-s}$.
Under this assumption, there are exactly two (distinct modulo $\ell^{e-s}$)
solutions, say $r_1$ and $r_2$, to the
congruence $r^2\equiv m\pmod{\ell^{e-2s_0}}$.
Therefore, if $Z$ is a solution to~\eqref{odd prime power cong with vanishing disc},
then either $Z=\ell^{s_0}(r_1+k_1\ell^{e-s})$ for some integer $k_1$ or
$Z=\ell^{s_0}(r_2+k_2\ell^{e-s})$ for some integer $k_2$.
In other words, $Z\equiv\ell^{s_0}r_1\pmod{\ell^{e-s_0}}$ or
$Z\equiv\ell^{s_0}r_2\pmod{\ell^{e-s_0}}$.
It is not hard to check that if $Z$ satisfies either of these two conditions, then
$Z$ is a solution to~\eqref{odd prime power cong with vanishing disc}.
There are exactly $2\ell^{e-(e-s_0)}=2\ell^{s_0}$ such values
for $Z$ modulo $\ell^e$.
\end{proof}

\begin{proof}[Proof of Lemma~\ref{compute factors of c}]
It is easily checked that $c_{N,f}(1)=1$.
If $n$ is odd, we observe that
\begin{equation*}
c_{N,f}(n)=\sum_{a\in(\Z/n\Z)^*}
\leg{a}{n}
\prod_{\ell\mid n}
\#C_N^{(\ell)}(a,n,f).
\end{equation*}
Now, suppose that $(m,n)=1$.  It follows that at least one of $m$ and $n$ must be odd.
Without loss of generality,
we assume that $n$ is odd.  Choose integers $m_0$ and $n_0$ so that
$4mm_0+nn_0=1$.
Then
\begin{equation*}
\begin{split}
c_{N,f}(n)c_{N,f}(m)&=
\sum_{a_1\in(\Z/n\Z)^*}
\leg{a_1}{n}
\prod_{\ell\mid n}
\#C_N^{(\ell)}(a_1,n,f)\\
&\quad\times
\sum_{\substack{a_2\in(\Z/4m\Z)^*\\ a_2\equiv 1\pmod 4}}
\leg{a_2}{m}
\mathcal S_2(m,a_2)
\prod_{\ell\mid m'}
\#C_N^{(\ell)}(a_2,m,f)\\
&=
\sum_{\substack{a_1\in(\Z/n\Z)^*,\\ a_2\in(\Z/4m\Z)^*\\ a_2\equiv 1\pmod 4}}
\leg{a_14mm_0+a_2nn_0}{nm}
\mathcal S_2(nm,a_14mm_0+a_2nn_0)\\
&\quad\times\prod_{\ell\mid nm'}
\#C_N^{(\ell)}(a_14mm_0+a_2nn_0,nm,f)\\
&=c_{N,f}(nm).
\end{split}
\end{equation*}
Hence $c_{N,f}(n)$ is multiplicative in $n$.

By~\eqref{defn of S_2}, we find that
\begin{equation*}
c_{N,f}(2^\alpha)
=\sum_{\substack{a\in(\Z/2^{\alpha+2}\Z)^*\\ a\equiv 1\pmod 4}}
\leg{a}{2^\alpha}\mathcal S_2(2^\alpha,a)
=2\sum_{\substack{a\in(\Z/2^{\alpha+2}\Z)^*\\ a\equiv 5\pmod 8}}
\leg{a}{2}^\alpha
=(-1)^\alpha 2^{\alpha}
\end{equation*}
for $\alpha\ge 1$.

We now consider the case when $\ell^\alpha$ is an odd prime power.
In view of Lemma~\ref{counting solutions}, it is natural to split
$c_{N,f}(\ell^\alpha)$ into two parts, writing
\begin{align*}
c_{N,f}^{(1)}(\ell^\alpha)&:=
\sum_{\substack{a\in(\Z/\ell^\alpha\Z)^*\\ \ell\dnd 4N+af^2}}
\leg{a}{\ell}^\alpha
\#C_N^{(\ell)}(a,\ell^\alpha,f),\\
c_{N,f}^{(0)}(\ell^\alpha)&:=
\sum_{\substack{a\in(\Z/\ell^\alpha\Z)^*\\ \ell\mid 4N+af^2}}
\leg{a}{\ell}^\alpha
\#C_N^{(\ell)}(a,\ell^\alpha,f)
\end{align*}
so that
\begin{equation}\label{split c for odd ell}
c_{N,f}(\ell^\alpha)
=c_{N,f}^{(1)}(\ell^\alpha)
+c_{N,f}^{(0)}(\ell^\alpha).
\end{equation}

We concentrate on $c_{N,f}^{(1)}(\ell^\alpha)$ first.
Applying Lemma~\ref{counting solutions}, we have
\begin{equation}\label{nonvanishing disc part of c}
\begin{split}
c_{N,f}^{(1)}(\ell^\alpha)
&=
\sum_{\substack{a\in(\Z/\ell^\alpha\Z)^*\\ \ell\dnd 4N+af^2\\ \ell\dnd (N-1)^2-af^2}}
\leg{a}{\ell}^\alpha
\#C_N^{(\ell)}(a,\ell^\alpha,f)+
\sum_{\substack{a\in(\Z/\ell^\alpha\Z)^*\\ \ell\dnd 4N+af^2\\ \ell\mid (N-1)^2-af^2}}
\leg{a}{\ell}^\alpha
\#C_N^{(\ell)}(a,\ell^\alpha,f)\\
&=\sum_{\substack{a\in(\Z/\ell^\alpha\Z)^*\\ \ell\dnd 4N+af^2\\ \ell\dnd (N-1)^2-af^2}}
\leg{a}{\ell}^\alpha
\left[1+\leg{4N+af^2}{\ell}\right]
+\sum_{\substack{a\in(\Z/\ell^\alpha\Z)^*\\ \ell\dnd 4N+af^2\\ \ell\mid (N-1)^2-af^2}}
\leg{a}{\ell}^\alpha\\
&=\ell^{\alpha-1}\left\{
\sum_{\substack{a\in(\Z/\ell\Z)^*\\ \ell\dnd 4N+af^2\\ \ell\dnd (N-1)^2-af^2}}
\leg{a}{\ell}^\alpha
\left[1+\leg{4N+af^2}{\ell}\right]
+\sum_{\substack{a\in(\Z/\ell\Z)^*\\ \ell\dnd 4N+af^2\\ \ell\mid (N-1)^2-af^2}}
\leg{a}{\ell}^\alpha
\right\}.
\end{split}
\end{equation}
In order to finish evaluating this sum, we split into cases.
First, assume that $\ell\mid f$.  Then the sum defining $c_{N,f}^{(1)}(\ell^\alpha)$ is
empty unless $\ell\dnd N$.
In this case, $\ell\mid (N-1)^2-af^2$ if and only if $\ell\mid N-1$.
Therefore, if $\ell\mid f$ and $\ell\dnd N$, then
\begin{equation}\label{c1 ell divides f}
\begin{split}
\frac{c_{N,f}^{(1)}(\ell^\alpha)}{\ell^{\alpha-1}}
&=\begin{cases}\displaystyle
\left[1+\leg{N}{\ell}\right]\sum_{a\in(\Z/\ell\Z)^*}\leg{a}{\ell}^\alpha
	&\text{if }\ell\dnd N-1,\\
\displaystyle
\sum_{a\in(\Z/\ell\Z)^*}\leg{a}{\ell}^\alpha	
	&\text{if }\ell\mid N-1
\end{cases}\\
&=\begin{cases}
(\ell-1)\left[1+\leg{N}{\ell}\right]&\text{if }2\mid\alpha\text{ and }\ell\dnd N-1,\\
\ell -1&\text{if }2\mid\alpha\text{ and }\ell\mid N-1,\\
0&\text{if }2\dnd\alpha
\end{cases}\\
&=\begin{cases}
(\ell-1)\left(1+\leg{N}{\ell}\leg{N-1}{\ell}^2\right)
	&\text{if }2\mid\alpha,\\
0&\text{if }2\dnd\alpha
\end{cases}\\
&=\#C_N^{(\ell)}(1,1,f)\begin{cases}
(\ell-1)&\text{if }2\mid\alpha,\\
0&\text{if }2\dnd\alpha
\end{cases}
\end{split}
\end{equation}
as $\#C_N^{(\ell)}(1,1,f)=\left(1+\leg{N}{\ell}\leg{N-1}{\ell}^2\right)$ in this case.

Now, suppose that $\ell\dnd f$.
Under this assumption, we note that $\leg{a}{\ell}=\leg{af^2}{\ell}$.
Picking up with equation~\eqref{nonvanishing disc part of c} and dividing through
by $\ell^{\alpha-1}$, we have
\begin{equation}\label{make c1 sum full}
\begin{split}
\frac{c_{N,f}^{(1)}(\ell^\alpha)}{\ell^{\alpha-1}}
&=
\sum_{\substack{a\in(\Z/\ell\Z)^*\\ \ell\dnd 4N+af^2\\ \ell\dnd (N-1)^2-af^2}}
\leg{af^2}{\ell}^\alpha
\left[1+\leg{4N+af^2}{\ell}\right]
+\sum_{\substack{a\in(\Z/\ell\Z)^*\\ \ell\dnd 4N+af^2\\ \ell\mid (N-1)^2-af^2}}
\leg{af^2}{\ell}^\alpha\\
&=
\sum_{\substack{a\in(\Z/\ell\Z)^*\\ \ell\dnd 4N+a\\ \ell\dnd (N-1)^2-a}}
\leg{a}{\ell}^\alpha
\left[1+\leg{4N+a}{\ell}\right]
+\sum_{\substack{a\in(\Z/\ell\Z)^*\\ \ell\dnd 4N+a\\ \ell\mid (N-1)^2-a}}
\leg{a}{\ell}^\alpha\\
&=
\sum_{\substack{a\in\Z/\ell\Z\\ a\not\equiv -4N\pmod\ell}}
\leg{a}{\ell}^\alpha
+\sum_{\substack{a\in\Z/\ell\Z\\  a\not\equiv (N-1)^2\pmod\ell}}
\leg{a}{\ell}^\alpha\leg{4N+a}{\ell}\\
&=\left[\sum_{b=1}^\ell\left(1+\leg{b}{\ell}\right)\leg{b-4N}{\ell}^\alpha\right]
-\leg{-N}{\ell}^\alpha-\leg{N^2-1}{\ell}^{2}
\end{split}
\end{equation}
upon completing the sums and making the change of variables $b=4N+a$.
One easily computes that
\begin{equation*}
\sum_{b=1}^\ell\left(1+\leg{b}{\ell}\right)\leg{b-4N}{\ell}^\alpha
=\begin{cases}
\ell-1 &\text{if }\ell\mid N,\\
\ell-1-\leg{N}{\ell} &\text{if }\ell\dnd N\text{ and }2\mid\alpha,\\
-1 &\text{if }\ell\dnd N\text{ and }2\dnd\alpha.
\end{cases}
\end{equation*}
For example, the third case is Exercise 1(b) of~\cite[p.~301]{MV:2007}.
Therefore,
\begin{equation*}
\frac{c_{N,f}^{(1)}(\ell^\alpha)}{\ell^{\alpha-1}}
=\begin{cases}
\ell-2
	&\text{if } \ell\mid N\text{ and }\ell\dnd f,\\
\ell-2-\leg{N}{\ell}-\leg{N^2-1}{\ell}^2
	&\text{if }\ell\dnd Nf\text{ and }2\mid\alpha,\\
-1-\leg{-N}{\ell}-\leg{N^2-1}{\ell}^2
	&\text{if }\ell\dnd Nf\text{ and }2\dnd\alpha.
\end{cases}
\end{equation*}

It remains to compute $c_{N,f}^{(0)}(\ell^\alpha)$.  We first consider the case when $\ell\dnd f$.
Note that the sum defining $c_{N,f}^{(0)}(\ell^\alpha)$ is empty unless $\ell\dnd N$ as well.
Thus, under this assumption, we have that
\begin{equation*}
\begin{split}
c_{N,f}^{(0)}(\ell^\alpha)
&=
\sum_{\substack{a\in(\Z/\ell^\alpha\Z)^*\\ \ell\mid 4N+af^2}}
\leg{af^2}{\ell}^\alpha
\#C_N^{(\ell)}(a,\ell^\alpha,f)\\
&=\sum_{\substack{a\in(\Z/\ell^\alpha\Z)^*\\ \ell\mid 4N+a}}
\leg{a}{\ell}^\alpha
\#C_N^{(\ell)}(a,\ell^\alpha,1)\\
&=\leg{-4N}{\ell}^\alpha
\sum_{\substack{a\in(\Z/\ell^\alpha\Z)^*\\ \ell\mid 4N+a}}
\#C_N^{(\ell)}(a,\ell^\alpha,1).
\end{split}
\end{equation*}
In order to evaluate this last sum, for each $a\in(\Z/\ell^{\alpha}\Z)^*$,
 we choose an integer representative in the range $-4N< a\le\ell^\alpha-4N$.  This ensures that
$0\le\nu_\ell(4N+a)\le\alpha$.
There is exactly one such choice of $a$ so that $\nu_\ell(4N+a)=\alpha$, namely
$a=\ell^\alpha-4N$.
For $1\le t\le\alpha-1$,  the number of such $a$ with $\nu_\ell(4N+a)=t$ is
$(\ell-1)\ell^{\alpha-t-1}$, but for only half of those values is
$\leg{(4N+a)/\ell^t}{\ell}=1$.  Perhaps, the easiest way to see this is to consider the
base-$\ell$ expansion of $4N+a$ for each $a$ in this range.
Therefore, applying Lemma~\ref{counting solutions} again, we have
\begin{equation*}
\begin{split}
\sum_{t=1}^\alpha
\sum_{\substack{0< 4N+a\le\ell^\alpha\\ \nu_\ell(4N+a)=t}}
\#C_N^{(\ell)}(a,\ell^\alpha,1)
&=\leg{N+1}{\ell}^2\ell^{\floor{\alpha/2}}
+\sum_{t=1}^{\floor{(\alpha-1)/2}}
\sum_{\substack{0<4N+a<\ell^\alpha\\ \nu_\ell(4N+a)=2t}}
\#C_N^{(\ell)}(a,\ell^\alpha,1)\\
&=\leg{N+1}{\ell}^2\ell^{\floor{\alpha/2}}
+\sum_{t=1}^{\floor{(\alpha-1)/2}}\frac{(\ell-1)\ell^{\alpha-2t-1}}{2}
	2\leg{N+1}{\ell}^2\ell^t\\
&=\leg{N+1}{\ell}^2\ell^{\floor{\alpha/2}}
	+\leg{N+1}{\ell}^2\ell^{\alpha-1}\left(1-\ell^{-\floor{(\alpha-1)/2}}\right)\\
&=\leg{N+1}{\ell}^2\left(
\ell^{\floor{\alpha/2}}+
\ell^{\alpha-1}\left(1-\ell^{-\floor{(\alpha-1)/2}}\right)
\right)\\
&=\leg{N+1}{\ell}^2\ell^{\alpha-1}.
\end{split}
\end{equation*}
Therefore, if $\ell\dnd fN$, then
\begin{equation*}
c_{N,f}^{(0)}(\ell^\alpha)
=\leg{-N}{\ell}^\alpha\leg{N+1}{\ell}^2\ell^{\alpha-1}.
\end{equation*}

We now compute $c_{N,f}^{(0)}(\ell^\alpha)$ in the case that $\ell\mid f$.
Note that, in this case, the sum defining $c_{N,f}^{(0)}(\ell^\alpha)$ is empty unless
$\ell\mid N$.
Note that $\nu_\ell(4N+af^2)\ge\min\{\nu_\ell(N),2\nu_\ell(f)\}$ with equality holding if
$\nu_\ell(N)\ne 2\nu_\ell(f)$.

First, suppose that $2\nu_\ell(f)<\nu_\ell(N)$.
Noting that $e=\nu_\ell(4\ell^\alpha f^2)>2\nu_\ell(f)$,
by Lemma~\ref{counting solutions} we have that
\begin{equation*}
\#C_N^{(\ell)}(a,\ell^\alpha,f)
=2\ell^{\nu_\ell(f)}=2\#C_N^{(\ell)}(1,1,f)
\end{equation*}
if and only if $\leg{a}{\ell}=\leg{(4N+af^2)/\ell^{2\nu_\ell(f)}}{\ell}=1$.
Hence,
\begin{equation*}
\begin{split}
c_{N,f}^{(0)}(\ell^\alpha)
&=\sum_{\substack{a\in(\Z/\ell^\alpha\Z)^*\\ \ell\mid 4N+af^2}}
\leg{a}{\ell}^\alpha\#C_N^{(\ell)}(a,\ell^\alpha,f)\\
&=2\#C_N^{(\ell)}(1,1,f)\sum_{\substack{a\in(\Z/\ell^\alpha\Z)^*\\ \leg{a}{\ell}=1}}\leg{a}{\ell}^\alpha\\
&=\#C_N^{(\ell)}(1,1,f)\ell^{\alpha-1}(\ell-1)
\end{split}
\end{equation*}
if $1< 2\nu_\ell(f)<\nu_\ell(N)$.

Now, suppose that $\nu_\ell(N)<2\nu_\ell(f)$. Since $e=\nu_\ell(4\ell^\alpha f^2)>\nu_\ell(N)$,
by Lemma~\ref{counting solutions}, we have that
\begin{equation*}
\begin{split}
\#C_N^{(\ell)}(a,\ell^\alpha,f)
&=\begin{cases}
2\ell^{\nu_\ell(N)/2}&\text{if } 2\mid\nu_\ell(N)\text{ and }\leg{N_\ell}{\ell}=1,\\
0&\text{otherwise}
\end{cases}\\
&=\#C_N^{(\ell)}(1,1,f)
\end{split}
\end{equation*}
for every $a\in(\Z/\ell^\alpha\Z)^*$.
Hence,
\begin{equation*}
\begin{split}
c_{N,f}^{(0)}(\ell^\alpha)
&=\sum_{\substack{a\in(\Z/\ell^\alpha\Z)^*\\ \ell\mid 4N+af^2}}
\leg{a}{\ell}^\alpha\#C_N^{(\ell)}(a,\ell^\alpha,f)\\
&=\#C_N^{(\ell)}(1,1,f)\sum_{a\in(\Z/\ell^\alpha\Z)^*}
\leg{a}{\ell}^\alpha\\
&=\#C_N^{(\ell)}(1,1,f)\begin{cases}
\ell^{\alpha-1}(\ell-1)&\text{if }2\mid \alpha,\\
0&\text{if }2\dnd \alpha
\end{cases}
\end{split}
\end{equation*}
if $1\le\nu_\ell(N)<2\nu_\ell(f)$.

Finally, consider the case when $2\nu_\ell(f)=\nu_\ell(N)$.
Let $r=\nu_\ell(f)$ and $s=\nu_\ell(N)$ and write
$f=\ell^rf_\ell$ and $N=\ell^sN_\ell$ with $(\ell,f_\ell N_\ell)=1$.
Then
\begin{equation*}
\begin{split}
c_{N,f}^{(0)}(\ell^\alpha)
&=\sum_{\substack{a\in(\Z/\ell^\alpha\Z)^*\\ \ell\mid 4N+af^2}}
\leg{a}{\ell}^\alpha\#C_N^{(\ell)}(a,\ell^\alpha,f)
=\sum_{a\in(\Z/\ell^\alpha\Z)^*}
\leg{af_\ell^2}{\ell}^\alpha\#C_N^{(\ell)}(a,\ell^\alpha,\ell^rf_\ell)\\
&=\sum_{a\in(\Z/\ell^\alpha\Z)^*}
\leg{a}{\ell}^\alpha\#C_N^{(\ell)}(a,\ell^\alpha,\ell^r).
\end{split}
\end{equation*}
To evaluate this last sum, for each $a\in(\Z/\ell^\alpha\Z)^*$, we choose an
integer representative in the range $-4N_\ell< a\le\ell^\alpha-4N_\ell$.
This ensures that $0\le\nu_\ell(4N_\ell+a)\le\alpha$, and hence that
$2r\le\nu_\ell(4N+a\ell^{2r})\le2r+\alpha$.
Similar to before, there is exactly one choice of $a$ such that
$\nu_\ell(4N_\ell+a)=\alpha$, namely $a=\ell^\alpha-4N_\ell$.
For $1\le t\le\alpha-1$, there are
$(\ell-1)\ell^{\alpha-t-1}$ choices with $\nu_\ell(4N_\ell+a)=t$, but for only half of
those values is $\leg{(4N_\ell+a)/\ell^t}{\ell}=1$.
Note that if $\ell\mid 4N_\ell+a$, then $\leg{a}{\ell}=\leg{-N_\ell}{\ell}$.
Therefore, if $1<\nu_\ell(N)=2\nu_\ell(f)$, then
\begin{equation*}
\begin{split}
c_{N,f}^{(0)}(\ell^\alpha)
&=\sum_{t=0}^\alpha\sum_{\substack{0< 4N_\ell+a\le \ell^\alpha\\ \nu_\ell(4N_\ell+a)=t}}
	\leg{a}{\ell}^\alpha\#C_N^{(\ell)}(a,\ell^\alpha,\ell^r)\\
&=\leg{-N_\ell}{\ell}^\alpha\ell^{\floor{(2r+\alpha)/2}}
	+\sum_{t=1}^{\floor{(\alpha-1)/2}}
	\leg{-N_\ell}{\ell}^\alpha\frac{(\ell-1)\ell^{\alpha-2t-1}}{2}2\ell^{r+t}\\
&\quad+\sum_{\substack{0<4N_\ell+a<\ell^\alpha\\ \leg{4N_\ell+a}{\ell}=1}}
	\leg{a}{\ell}^\alpha 2\ell^r\\
&=\leg{-N_\ell}{\ell}^\alpha\ell^{r+\floor{\alpha/2}}
	+\leg{-N_\ell}{\ell}^\alpha\ell^{\alpha-1+r}\left(1-\ell^{-\floor{(\alpha-1)/2}}\right)\\
&\quad+\ell^{r+\alpha-1}
	\sum_{a\in\Z/\ell\Z}\leg{a}{\ell}^\alpha\left(1+\leg{4N_\ell+a}{\ell}\right)\\
&=\leg{-N_\ell}{\ell}\ell^r\left[\ell^{\floor{\alpha/2}}
	+\ell^{\alpha-1}\left(1-\ell^{-\floor{(\alpha-1)/2}}\right)\right]\\
&\quad+\ell^{r+\alpha-1}
	\sum_{b\in\Z/\ell\Z}\left(1+\leg{b}{\ell}\right)\leg{b-4N_\ell}{\ell}^\alpha\\
&=
\ell^{r+\alpha-1}
\begin{cases}
\leg{-N_\ell}{\ell}+\ell-1-\leg{N_\ell}{\ell}&\text{if }2\mid\alpha,\\
\leg{-N_\ell}{\ell}-1&\text{if }2\dnd\alpha
\end{cases}\\
&=\#C_N^{(\ell)}(1,1,f)\ell^{\alpha-1}
\begin{cases}
\leg{-N_\ell}{\ell}+\ell-1-\leg{N_\ell}{\ell}&\text{if }2\mid\alpha,\\
\leg{-N_\ell}{\ell}-1&\text{if }2\dnd\alpha.
\end{cases}
\end{split}
\end{equation*}
The lemma now follows by combining our computations for $c_{N,f}^{(0)}(\ell^\alpha)$
and $c_{N,f}^{(1)}(\ell^\alpha)$.

\end{proof}

\bibliographystyle{alpha}
\bibliography{references}
\end{document}